\documentclass[12pt]{article}
\usepackage[top=2cm,bottom=2cm,left=2cm,right=2cm]{geometry}
\usepackage{enumerate}
\usepackage{euscript}
\usepackage{graphicx}
\usepackage{amssymb}
\usepackage{amsmath}
\usepackage{amsthm}
\usepackage{dsfont}
\usepackage[colorinlistoftodos]{todonotes}
\usepackage{color}

\theoremstyle{plain}
\newtheorem{lem}{Lemma}
\newtheorem{rem}{Remark}

\newtheorem{theorem}{Theorem}

\def\beh#1{{\color{black}#1}}

\theoremstyle{definition}
\newtheorem{definition}{Definition}
\newtheorem{assumption}{Assumption}
\newcommand{\E}{\mathrm{E}}
\newcommand{\Var}{\mathrm{Var}}
\newcommand{\Tr}{\mathrm{Tr}}

\def\fb{\mathbf{f}}
\def\R{\mathbb{R}}
\def\Wb{\mathbf{W}}
\def\Z{\mathbb{Z}}
\def\zb{\mathbf{z}}
\def\zx{\mathbf{x}}
\def\zxk{\hat{\zx}^{\tau,\boldsymbol{\kappa}}}
\def\ze{\mathbf{e}}
\def\zq{\mathbf{q}}
\def\zz{\mathbf{0}}
\def\zhf{\mathbf{\hat f}}
\def\zbx{\mathbf{X}}
\def\ab{\mathbf{a}}
\def\bb{\mathbf{b}}
\def\vb{\mathbf{v}}

\begin{document}
	\title{Non-Convex Distributed Optimization}
	\author{Tatiana Tatarenko\thanks{Department of Control Theory and Robotics, TU Darmstadt, Germany. Email: tatarenk@rtr.tu-darmstadt.de} and Behrouz Touri\thanks{Department of Electrical, Computer, and Energy Engineering, Boulder, USA. Email: behrouz.touri@colorado.edu}}
	\date{}
\maketitle

\begin{abstract}
We study distributed non-convex optimization on a time-varying multi-agent network. Each node has access to its own smooth local cost function, and the collective goal is to minimize the sum
of these functions. {The perturbed push-sum algorithm was previously used for convex distributed optimization.}
We generalize the result obtained {for the convex case} to the case of non-convex functions. Under some additional technical assumptions on the gradients we prove the convergence of the distributed push-sum algorithm to some critical point of the objective function. By utilizing perturbations on the update process, we show the almost sure convergence of the perturbed dynamics to a local minimum of the global objective function, \beh{if the objective function has no saddle points}. Our analysis shows that this perturbed procedure converges at a rate of $O(1/t)$.
\end{abstract}

\section{Introduction}\label{sec:intro}
Due to emergence of the distributed networked systems, distributed multi-agent optimization problems have gained a lot of attention over the recent years. These systems include, but are not limited to, robust
sensor network control~\cite{Rabbat}, signal processing~\cite{touri2010infinite,olshevskythes}, power control~\cite{RamVN09}, network routing~\cite{Neglia}, machine learning~\cite{ML}, opinion dynamics \cite{mohajer2012convergence,touri2011discrete,etesami2013termination}, and spectrum access coordination~\cite{LiHan}. In all these areas a number of agents, which can be represented by nodes over communication graphs, are required to optimize a global objective without any centralized computation unit and taking only the local information into account.

Most of the theoretical work on the topic was devoted to optimization of a sum of convex functions, where the assumption on subgradient existence is essentially used. We refer the reader to the following papers that consider various settings on network topology, asynchronous updating rules, and noisy communication~\cite{Chen, gharesifard2014distributed, LobelO11, LobelOF11, NedicO09, touri2015continuous, touri2010asynchronous}.
However, in many applications it is crucial to have an efficient solution for \emph{non-convex optimization} problems~\cite{Baras,Scutari,nonconDOP}. In particular, resource allocation problems with non-elastic traffic were studied in~\cite{nonconDOP}.
Such applications cannot be modeled
by means of concave utility functions, that renders a problem a non-convex optimization.
{
Another example of non-convex optimization can be found in machine learning, where the independent component analysis~\cite{ML} needs to be performed for big data by a network of processors. The cumulative loss function is represented by the sum of errors of the model on a cross-validation sample. Due to the increasing size of available data and problem dimensionality, the cumulative loss is to be minimize in a distributed manner, where each error function is assigned to a different processor, and processors can exchange the information only with their local neighbours in a network~\cite{NIPS2012_4566}.
Since these problems refer to a non-convex optimization, some more sophisticated solution techniques are needed to handle them.
}

In~\cite{Baras}, the authors proposed a distributed algorithm for non-convex constrained optimization based on a first order numerical method. However, the convergence to local minima is guaranteed only under the assumption that the communication topology is time-invariant, the initial values of the agents are close enough to a local minimum, and a sufficiently small step-size is used. An approximate dual subgradient algorithm over time-varying network topologies was proposed in~\cite{Martinez}. This algorithm converges to a pair of primal-dual solutions to the approximate
problem
given the Slater's condition for a constrained optimization problem, and under the assumption that the optimal solution set of the dual limit is singleton.
{Recently, ~\cite{ADMM, ADMM1}  provided an analysis of the alternating direction method of multipliers and the alternating direction penalty method applied to non-convex settings. Convergence to a primal feasible point under mild conditions was proved. Some additional conditions were introduced to ensure the first order necessary conditions for local optimality for a limit point.}

{In this paper we study the push-sum protocol, which is a special case of consensus dynamics, applied to distributed optimization. The idea behind such consensus-based algorithm is the combination of a consensus dynamics and a gradient descent along the agent's local function.
The push-sum algorithm was initially introduced in~\cite{16A2} and used in~\cite{Tsianos2012} for distributed optimization.}
 {The main advantage of the push-sum protocol is that it can be utilized to overcome the restrictive assumptions on the communication graph structure such as double stochastic communication matrices~\cite{16A2, Tsianos2012}.}
The work~\cite{A1} studied this algorithm over \emph{time-varying communication} in the case of well-behaved convex functions. Similar to~\cite{A1}, we focus on the perturbed push-sum algorithm. In our work we relax the assumption on the convexity and apply this algorithm to a more general case.

The main contribution of this paper is the proof \textit{of convergence of the push-sum distributed optimization algorithm to a critical point of a sum of non-convex smooth functions} under some general assumptions on the underlying functions and general connectivity assumptions on the underlying graph. Moreover, a \textit{perturbation of the algorithm allows us to use the algorithm to search local optima} of this sum. It means that the new stochastic procedure steers every node to some local minimum of the objective function from any initial state. In our analysis we use the result on stochastic recursive approximation procedures extensively studied in~\cite{NH}. We also present the analysis of the convergence rate for the procedure under consideration.

  Some other works dealt with the distributed stochastic approximation with applications to non-convex optimization~\cite{Bianchi, NonDS}. The authors in~\cite{Bianchi} demonstrated convergence of the distributed stochastic approximation procedure under a specific communication protocol to the set of critical points in unconstrained optimization. In~\cite{NonDS} the authors generalized this procedure to the case of constrained optimization by considering a projected version of the gradient step and proved convergence of this algorithm to Karush-Kuhn-Tucker (KKT) points. However, critical points and KKT points represent only a necessary condition of local minima. Moreover, the communication protocol proposed in~\cite{Bianchi, NonDS} relied on a double-stochastic communication matrix on average, and, thus, required a double-stochastic matrix in the case of deterministic communication, which may also restrict the range of potential applications. In contrast to the papers above, this work presents the push-sum 
communication  protocol requiring agents to know only the number of their out-neighbors. Moreover, this protocol allows the stochastic gradient step to lead the system not only toward some critical point, but to a local minimum almost surely \beh{in the absence of saddle points.  Note that none of the existing first order methods can guarantee avoidance of saddle points without further assumptions on the second derivative \cite{borkar2008stochastic,ljung1978strong,dauphin}. As the push-sum algorithm presented in this paper does not rely on assumptions on second order derivatives, the procedure converges either to a local minimum or to a saddle point in the presence of a latter one.}

The paper is organized as follows. In Section~\ref{sec:dop} we formulate the distributed optimization problem of interest and introduce the push-sum algorithm that we adopt further to solve this problem. Section~\ref{sec:prelim} presents some important preliminaries for the theoretical analysis. Section~\ref{sec:mainresults} contains the main results on the convergence to critical points and local optima. Section~\ref{sec:proofs} provides the proofs of the main results. Section~\ref{sec:sim} contains an illustrative example of implementation of the proposed algorithm for a non-convex optimization problem. Finally, Section~\ref{sec:conclude} concludes the paper.

\textbf{Notation:} We will use the following notations throughout this paper: We denote the set of integers by $\Z$ and the set of non-negative integers by $\Z^+$. For the metric $\rho$ of a metric space $(X,\rho(\cdot))$ and a subset $B\subset X$, we let $\rho(x,B)=\inf_{y\in B}\rho(x,y)$. We denote the set $\{1,\ldots,n\}$ by $[n]$. We use boldface to distinguish between the vectors in a multi-dimensional space and scalars. We denote the dot product of two vectors $\ab$ and $\bb$ by $(\ab,\bb)$. Throughout this work, all time indices such as $t$ belong to $\Z^+$. For vectors $\vb_i\in X^d$, $i\in[n]$, of elements in some vector space $X$ (over $\R$), we let $\bar{\vb}=\frac{1}{n}\sum_{i=1}^n\vb_i$.

\section{Distributed Optimization Problem}\label{sec:dop}
In this section, we provide the problem formulation.

\subsection{Problem Formulation}\label{subsec:prform}
Consider a network of $n$ agents. At each time $t$, node $i$ can only communicate to  its  out-neighbors  in  some  directed  graph $G(t)$, where the graph $G(t)$ has the vertex  set $[n]$ and the edge set $E(t)$. We introduce the following standard definition for the sequence $G(t)$.
\smallskip

\begin{definition}
	We say that a sequence of graphs $\{G(t)\}$ is \textit{$S$-strongly connected}, if for any time $t\geq 0$, the graph
	\[G(t:t+S)=([n],E(t)\cup E(t+1)\cup \cdots \cup E(t+S-1)),\]
	is strongly connected. In other words, the union of the graphs over every $S$ time intervals is strongly connected.
\end{definition}
\smallskip

The assumption on the $S$-strongly connected sequence of communication graphs has been used in many prior works~\cite{NedicO09,A1,tsitsiklis1986} to ensure enough mixing of information among the agents.

We use $N^{in}_i(t)$ and $N^{out}_i(t)$ to denote the in- and out-neighborhoods of node $i$ at time $t$. Each node $i$ is always considered to be an in- and out-neighbor of itself. We
use $d_i(t)$ to denote the out-degree of node
$i$, and we assume that every node $i$ knows its out-degree at every time $t$.

The goal of the agents is to solve distributively the following minimization problem:
\begin{align}\label{eq:DOP}
\min_{\zb\in\mathbb{R}^d}  F(\zb) = \sum_{i=1}^{n}F_i(\zb).
\end{align}
The essential assumption in many distributed optimization problems is that the function $F_i:\R^d\to\R$ is only available to agent $i$.

In this paper we introduce the following assumption on the gradients of $F_i$:
\smallskip

\begin{assumption}\label{assum:A1}
 Each function $F_i(\zb)$, $i\in [n]$, has the gradient $\fb_i(\zb)=\nabla F_i(\zb)$ such that the norm of $\fb_i(\zb)$ is uniformly bounded by some finite constant $\alpha\ge 0$ for each $i\in [n]$, i.e.\  $\|\fb_i(\zb)\|\leq \alpha$ for all $\zb\in \R^d$ and $i\in [n]$.
\end{assumption}
\smallskip

Note that Assumption~\ref{assum:A1} is a standard assumption that is made in many of the previous works on subgradient methods (including~\cite{A1}). {We will further assume that a solution of~\eqref{eq:DOP} exists, and the set of critical points of the objective function $F(\zb)$, i.e. the set of points $\zb$ such that $\nabla F(\zb)=\zz$, is a union of finitely many connected components.}

\section{Preliminaries}\label{sec:prelim}
{
This section is devoted to preliminaries of our main approach to study the problem~\eqref{eq:DOP}.

\subsection{Stochastic Recursive Approximation Procedure}\label{subsec:prelsrp}
In the following sections we propose an approach to the distributed optimization problem~\eqref{eq:DOP} based on the convergence of stochastic approximation procedures analyzed in~\cite{NH}.
For this purpose, we consider a $d$-dimensional non-autonomous nonlinear stochastic process $\{\zbx(t)\}_t$ taking values in $\mathbb{R}^d$ and being expressed in the form
\begin{align}\label{eq:basic}
\zbx(t+1)=\zbx(t)-a(t+1)\tilde{\fb}(t,\zbx(t))-a(t+1)\Wb(t+1),
\end{align}
where $\zbx(0)=\zx_0\in \R^d$, $\tilde{\fb}(t,\zbx(t))=\fb(\zbx(t))+\zq(t,\zbx(t))$ such that
$\fb: \mathbb{R}^d\to\mathbb{R}^d$, $\zq(t,\zbx(t)): \mathbb{R}\times\mathbb{R}^d\to\mathbb{R}^d$,
$\{\Wb(t)\}_t$ is a sequence of Markov random vectors, i.e.
\[\Pr \{\Wb(t+1)|\zbx(t),\zbx(t-1),\ldots,\zbx(0)\}=\Pr\{\Wb(t+1)|\zbx(t)\},\]
 and $a(t)>0$ is a step-size parameter. }

We proceed by presenting some important results on the convergence of stochastic approximation procedures from~\cite{NH}.
Let $\{x(t)\}_t$, $t\in \Z^+$, be a general Markov process on some state space $X\subset \R^d$. The transition kernel of this chain, namely $\Pr\{x(t+1)\in\Gamma| x(t)=x\}$, is denoted by $P(t,x,t+1,\Gamma)$, $\Gamma\subseteq X$.
\smallskip

\begin{definition}\label{def:def1}
The operator $L$ defined on the set of measurable functions $V:\Z^+\times X\to \R$, $x\in X$, by
\begin{align*}
LV(t,x)&=\int{P(t,x,t+1,dy)[V(t+1,y)-V(t,x)]}\cr
&=\E[V(t+1,x(t+1))\mid x(t)=x]-V(t,x).
\end{align*}
is called the \emph{generating operator} of a Markov process $\{x(t)\}_t$.
\end{definition}
\smallskip

In the case of a deterministic process $\{x(t)\}_t$ on $X$, the generating operator reduces to the difference equation
\begin{align*}
 LV(t,x)=V(t+1,x(t+1))-V(t,x(t)).
\end{align*}

Let $B$ be a subset of $X$, $U_{\epsilon}(B)$ be its $\epsilon$-neighborhood, i.e. $U_{\epsilon}(B)=\{x: \rho(x,B)<\epsilon\}$.
Let $V_{\epsilon}(B)=X\setminus U_{\epsilon}(B)$ and $U_{\epsilon,R}(B)=V_{\epsilon}(B)\cap \{x:\|x\|<R\}$.
\smallskip

\begin{definition}\label{def:def2}
The function $\phi(t,x)$ is said to belong to class $\Phi(B)$ if

1) $\phi(t,x)\ge 0$ for all $t\in \Z^+$ and $x\in X$,

2) for all $R>\epsilon >0$ there exists some $Q=Q(\epsilon, R)$ such that
\begin{align*}
 \inf_{t\ge Q, x\in U_{\epsilon,R}(B)}\phi(t,x)>0.
\end{align*}
\end{definition}
\smallskip

Now we turn back to the process~\eqref{eq:basic} and formulate the following theorem, which was proven in~\cite{NH} (Theorem 2.7.3).
\smallskip

\begin{theorem}\label{th:th1}
  Consider the dynamics defined by~\eqref{eq:basic}, where the random term $\Wb(t)$ has zero mean and bounded variance. Let us suppose that the set $B=\{\zx: \fb(\zx)=\zz\}$ be a {union of finitely many connected components. Assume that there exist a set $B$, sequences $a(t)$, $g(t)$}, and some function $V(t,\zx): \mathbb{Z}^+\times\mathbb{R}^d\to\mathbb R$ that satisfy the following assumptions:

 \begin{enumerate}[(1)]
 	\item\label{th1:cond1}  For all $t\in \Z^+$, $\zx\in \mathbb{R}^d$, we have $V(t,\zx)\ge 0 $ and $\inf_{t\ge0}V(t,\zx)\to\infty$ as $\|\zx\|\to\infty$,
	\item\label{th1:cond2}  $LV(t,\zx)\le -a(t+1)\phi(t,\zx) + g(t)(1+V(t,\zx))$, where  $\phi\in\Phi(B)$, $g(t)\ge 0$, $\sum_{t=0}^{\infty}g(t)<\infty$,
	\item\label{th1:cond3}  $a(t)>0$, $\sum_{t=0}^{\infty} a(t)= \infty$, and $\sum_{t=0}^{\infty} a^2(t)<\infty$, and there exists $c(t)$ such that $\|\zq(t,\zbx(t))\|\le c(t)$ a.s.\ for all $t\ge0$ and $\sum_{t=0}^{\infty}a(t+1)c(t)<\infty$, and
	\item\label{th1:cond4} $\sup_{t\ge0,\|\zx\|\le r}\|\tilde{\fb}(t,\zx)\|<\infty$ {for any $r$.}
\end{enumerate}
  {Then the process~\eqref{eq:basic} converges almost surely either to a point from $B$ or to the boundary of one of its connected components}.
\end{theorem}
\smallskip

\begin{rem}
Note that Theorem~\ref{th:th1} holds in the deterministic case of the process~\eqref{eq:basic}, when $\Wb(t)\equiv \zz$ for all $t\ge0$.
\end{rem}

\subsection{Perturbed Push-Sum Algorithm for Distributed Optimization}\label{subsec:pertps}
Here we discuss the general push-sum algorithm initially proposed in~\cite{16A2}, applied in~\cite{Tsianos2012} to distributed optimization of convex functions, and analyzed in~\cite{A1}.
For this algorithm, we have a network of $n$ agents introduced in Section~\ref{subsec:prform}.
According to the general push-sum protocol, at every moment of time $t\in \Z^+$ each node $i$ maintains   vector   variables $\zb_i(t)$, $\zx_i(t)$, $\boldsymbol w_i(t)\in \mathbb{R}^d$, as  well  as  a  scalar  variable $y_i(t)$ with $y_i(0)=1$ for all $i\in[n]$. These quantities are updated according to the following rules:
\begin{align}\label{eq:gps}
\boldsymbol w_i(t+1)&=\sum_{j\in N^{in}_i(t)}\frac{\zx_j(t)}{d_j(t)},\cr
y_i(t+1)&=\sum_{j\in N^{in}_i(t)}\frac{y_j(t)}{d_j(t)},\cr
\zb_i(t+1)&=\frac{\boldsymbol w_i(t+1)}{y_i(t+1)},\cr
\zx_i(t+1)&=\boldsymbol w_i(t+1)+\ze_i(t+1),
\end{align}
where $\ze_i(t)$ is some $d$-dimensional, possibly random, perturbation at time $t$.
\smallskip
\begin{theorem}~\cite{A1}\label{th:th2}
Consider  the  sequences $\{\zb_i(t)\}_t$, $i\in [n]$, generated  by the algorithm~\eqref{eq:gps}. Assume that the graph sequence
$\{G(t)\}$ is $S$-strongly connected. Then for some constants $\delta$ and $\lambda$  satisfying $\delta\ge\frac{1}{n^{nS}}$ and  $\lambda\le \left(1-\frac{1}{n^{nS}}\right)^{1/S}$ for all $i\in[n]$ we have
\begin{align*}
 &\left\|\zb_i(t+1)-\bar{\zx}(t) \right\| \cr
 &\le \frac{8}{\delta}\left(\lambda^t\sum_{j=1}^{n}\|\zx_j(0)\|_1+\sum_{s=1}^{t}\lambda^{t-s}\sum_{j=1}^{n}\|\ze_j(s)\|_1\right).
\end{align*}

Moreover, if $\{a(t)\}$ is a non-increasing positive scalar sequence with \newline $\sum_{t=1}^{\infty}a(t)\|\ze_i(t)\|_1<\infty$ a.s. for all $i\in [n],$ then
\begin{align*}
 &\sum_{t=0}^{\infty}a(t+1)\left\|\zb_i(t+1)-\bar{\zx}(t) \right\|\cr
 &\le\sum_{t=0}^{\infty}\frac{8a(t+1)}{\delta}\left(\lambda^t\sum_{j=1}^{n}\|\zx_j(0)\|_1+\sum_{s=1}^{t}\lambda^{t-s}\sum_{j=1}^{n}\|\ze_j(s)\|_1\right)\cr
 &\qquad\qquad\qquad\qquad\qquad\qquad<\infty \mbox{ almost surely for all $i$,}
\end{align*}
 where $\|\cdot\|_1$ is the $l^1$-norm in $\mathbb{R}^d$.
\end{theorem}
\smallskip

Note that the above theorem implies that, if $\lim_{t\to\infty}\|\ze(t)\|_1=0$ a.s., then 
\begin{align*}
 \lim_{t\to\infty}\left\|\zb_i(t+1)-\bar{\zx}(t) \right\|=0 \quad \mbox{ a.s. for all }i.
\end{align*}

In other words, under some assumptions on the perturbations $\ze_i(t)$ one can guarantee that in the push-sum algorithm all $\zb_i(t)$ track the average state $\bar{\zx}(t)$.

{Similar to~\cite{A1}, we adopt the push-sum algorithm~\eqref{eq:gps} to the distributed optimization problem~\eqref{eq:DOP} by letting $\ze_i(t+1) = -a(t+1)\fb_i(\zb_i(t+1))$ or $\ze_i(t+1) ~= -a(t+1)(\fb_i(\zb_i(t+1))+\Wb_i(t+1))$, where $\Wb_i(t+1)$ is a random vector whose entries are independent random variables with zero mean and bounded variance.}

Under Assumption~\ref{assum:A1} and given that $\lim_{t\to\infty}a(t)=0$, in both cases
\begin{align*}
 \lim_{t\to\infty}\|\ze_i(t)\|_1=0 \quad \mbox{ a.s. for all }i.
\end{align*}
Thus, in long run, the nodes' variables $\zb_i(t+1)$ will track the average state $\bar{\zx}(t) = \frac{1}{n}\sum_{j=1}^{n}\zx_j(t)$ (see Theorem~\ref{th:th2}).
Hence, one can expect that the long run behavior of the iterations in~\eqref{eq:gps} with $\ze_i(t+1) = -a(t+1)\fb_i(\zb_i(t+1))$ is the same as the behavior of the gradient descent iteration:
\begin{align*}
\bar{\zx}(t+1)=\bar{\zx}(t)-a(t+1)\frac1n\sum_{i=1}^{n}\fb_i(\bar{\zx}(t)),
\end{align*}
whereas the long run behavior of the iterations in~\eqref{eq:gps} with $\ze_i(t+1) = -a(t+1)(\fb_i(\zb_i(t+1))+\Wb_i(t+1))$ is equivalent to the behavior of the Robbins-Monro iteration~\cite{RMP}:
\begin{align*}
\bar{\zx}(t+1)=\bar{\zx}(t)-a(t+1)\left(\frac1n\sum_{i=1}^{n}\fb_i(\bar{\zx}(t))+\bar{\Wb}(t+1)\right).
\end{align*}

\section{Main Results}\label{sec:mainresults}

Here, we discuss the main results of this work. We present the proofs of these results in the subsequent sections.

\subsection{Deterministic Procedure: Convergence to Critical Points}\label{subsec:detpr}
First, let us recall the distributed optimization algorithm using the push-sum algorithm. At every moment of time $t\in \Z^+$ each node $i$ maintains   vector   variables $\zb_i(t)$, $\zx_i(t)$, $\boldsymbol w_i(t)\in \mathbb{R}^d$, as  well  as  a  scalar  variable $y_i(t)$: $y_i(0)=1$ for all $i\in[n]$. These quantities are updated according to \eqref{eq:gps} and
\begin{align}\label{eq:ps}
\zx_i(t+1)&=\boldsymbol w_i(t+1)-a(t+1)\fb_i(\zb_i(t+1)).
\end{align}

The process above is a special case of the \emph{perturbed push-sum} algorithm, whose background and properties are discussed in Section~\ref{subsec:pertps}.
According to~\eqref{eq:ps}, the average state $\bar{\zx}(t)$ follows the dynamics 
\begin{align}\label{eq:runavdp}
\bar{\zx}(t+1)&=\bar{\zx}(t)-a(t+1)\frac1n\sum_{i=1}^{n}\fb_i({\zb_i(t+1)}),
\end{align}
which can be rewritten as
\begin{align}\label{eq:runavdp1}
&\bar{\zx}(t+1)=\bar{\zx}(t)\\\nonumber
&\quad-a(t+1)\left(\fb(\bar{\zx}(t))+\left[\frac1n\sum_{i=1}^{n}\fb_i({\zb_i(t+1)})-\fb(\bar{\zx}(t))\right]\right),
\end{align}
where $\fb(\zb)=\frac1n\sum_{i=1}^{n}\fb_i(\zb)=\frac1n\nabla F(\zb)$.
If we denote $\frac1n\sum_{i=1}^{n}\fb_i({\zb_i(t+1)})-\fb(\bar{\zx}(t))$ by $\zq(t,\bar{\zx}(t))$, the process~\eqref{eq:runavdp1} becomes a deterministic version of the process~\eqref{eq:basic} ($\Wb(t)\equiv \zz$ for any $t$).
It is shown in~\cite{A1} that if functions $F_i(\zb_i)$ are convex functions satisfying Assumption~\ref{assum:A1}, and there exists a solution of~\eqref{eq:DOP}, then the process~\eqref{eq:runavdp} converges to an optimizer of the function $F$, given an appropriate choice of step-size sequence $a(t)$.

Let the gradients $\fb_i(\zx)$, $i\in [n]$, satisfy the following assumption:
\smallskip

\begin{assumption}\label{assum:A2}
For each $i\in [d]$, $\fb_i$ is Lipschitz continuous on $\mathbb R^d$, i.e. there exists a constant $l_i\ge 0$ such that $\|\fb_i(\zx_1)-\fb_i(\zx_2)\|\le l_i\|\zx_1-\zx_2\|$ for any $\zx_1,\zx_2\in\mathbb R^d$.
\end{assumption}
\smallskip

Further we need the following assumption on the behavior of the objective function $F(\zb)$, when $\|\zb\|\to\infty$.

\begin{assumption}\label{assum:A3}
 $F(\zb)$ is coercive, i.e $\lim_{\|\zb\|\to\infty}F(\zb)\to\infty$\footnote{Since we also require the boundedness of $\nabla F$ (Assumption~\ref{assum:A1}), the function $F(\zb)$ is assumed to have a linear behavior as $\zb\to\infty$.}.
\end{assumption}

{
Let us denote the set of critical points of $F$ by $B$, i.e.
\begin{align*}
B = \{\zb\in \mathbb R^d: \fb(\zb)=\zz\}.
\end{align*}
We use $B^{''}$ to represent its refinement to the set of local minima, i.e.
\begin{align*}
B^{''} = \{\zb\in B: \mbox{ there exists } &\delta: \fb(\zb)\le \fb(\zb'), \cr
&\mbox{ for any } \zb': \|\zb'-\zb\|<\delta\},
\end{align*}
and represent the rest of the critical points by $B^{'}=B\setminus B^{''}$.
}
Now we are ready to formulate the main result of this section.
\smallskip

\begin{theorem}\label{th:th3}
 Let the functions $F$ and $\fb_i$, $i\in[n]$, satisfy Assumptions~\ref{assum:A1}-\ref{assum:A3}. Let $\{a(t)\}$ be a positive and non-increasing step-size sequence such that $\sum_{t=0}^{\infty}a(t)=\infty$ and $\sum_{t=0}^{\infty}a^2(t)<\infty$.
 { Then the average state $\bar{\zx}(t)$ and $\zb_i(t)$ (see~\eqref{eq:runavdp1}) for the dynamics~\eqref{eq:ps} with $S$-strongly connected graph sequence $\{G(t)\}$ converge either to a point from the set $B$ or to the boundary of one of its connected components.}
\end{theorem}
\smallskip

Note that the deterministic process~\eqref{eq:runavdp1} suffers from the inherent property of gradient-based optimization methods as the derivative cannot distinguish between various types of critical points.

\subsection{Perturbed Procedure: Convergence to Local Minima}\label{subsec:pertpr}
In this subsection, we overcome the weakness mentioned at the end of the previous subsection by modifying the deterministic process~\eqref{eq:runavdp1}. We add some noise to the iterative process that will render it as a Markov chain in $\mathbb{R}^d$. { The idea of adding a noise to the local optimization step is similar
to one of simulated annealing ~\cite{bertsimas1993}. Here, the idea is that the non-local minima points have unstable directions (with respect to the gradient descent dynamics) that can be explored using the additional noise and hence, the algorithm will never get stuck at the non-local minima critical points.} {We will show that this noise together with an appropriate choice of the step-size sequence $a(t)$ allows the algorithm to converge to a local optimum and not to become stuck in a suboptimal critical point.} Let us consider the following variation of~\eqref{eq:ps}:
\begin{align}\label{eq:pps}
\boldsymbol w_i(t+1)&=\sum_{j\in N^{in}_i(t)}\frac{\zx_j(t)}{d_j(t)},\cr
y_i(t+1)&=\sum_{j\in N^{in}_i(t)}\frac{y_j(t)}{d_j(t)},\cr
\zb_i(t+1)&=\frac{\boldsymbol w_i(t+1)}{y_i(t+1)},\cr
\zx_i(t+1)&=\boldsymbol w_i(t+1)\!-a(t+1)\fb_i(\zb_i(t+1))\cr
&\qquad\qquad\quad-a(t+1)\Wb_i(t+1),
\end{align}
where the functions $\fb_i$ are the subgradient functions and $\{\Wb_i(t)\}_t$ is the sequence of independently and identically distributed (i.i.d.) random vectors taking values in $\mathbb{R}^d$ and satisfying the following assumption:
\smallskip

\begin{assumption}\label{assum:A4}
  The entries $W_i^k(t)$, $W_i^j(t)$ of each $\Wb_i(t)$, $i\in[n]$, are independent, $\E(W_i^k(t))=0$, and $\Var(W_i^k(t))=1$ for all $t\in \Z^+$ and $k,j=1,\ldots, d$.
\end{assumption}
\smallskip

The procedure above implies the following update for the average state:
\begin{align}\label{eq:runavpp}
&\bar{\zx}(t+1)=\bar{\zx}(t)\\\nonumber
&\qquad-a(t+1)\left(\fb(\bar{\zx}(t))+\left[\frac1n\sum_{i=1}^{n}\fb_i({\zb_i(t+1)})-\fb(\bar{\zx}(t))\right]\right)\cr
&\qquad-a(t+1)\bar{\Wb}(t+1).
\end{align}

Under Assumption~\ref{assum:A4}, the process~\eqref{eq:runavpp} is a Markov chain on the space $\mathbb{R}^d$ that is a special case of the process~\eqref{eq:basic}.
\smallskip

\begin{theorem}\label{th:th4}
Let Assumptions~\ref{assum:A1}-\ref{assum:A4} hold. Let $\{a(t)\}$ be a positive and non-increasing step-size sequence with $\sum_{t=0}^{\infty}a(t)=\infty,$ $\sum_{t=0}^{\infty}a^2(t)<\infty$.
   { Then the average state $\bar{\zx}(t)$ and $\zb_i(t)$ defined by~\eqref{eq:pps} and~\eqref{eq:runavpp} with $S$-strongly connected graph sequence $\{G(t)\}$ converge either to a point from the set $B$ or to the boundary of one of its connected components.}
\end{theorem}
\smallskip

The above result is an analogue of Theorem~\ref{th:th3} adapted for~\eqref{eq:pps} and ~\eqref{eq:runavpp}. It is clear that Theorem~\ref{th:th4} does not guarantee the convergence of the process~\eqref{eq:runavpp} to some local minimum of the objective function $F$. However, we will prove further the following theorem claiming that the process~\eqref{eq:runavpp} cannot converge to a critical point that is not some local minimum of the objective function $F(\zb)$, given an appropriate choice of step-size sequence $a(t)$, Assumptions~\ref{assum:A1}-\ref{assum:A4}, and
\begin{assumption}\label{assum:A5A}
 For any point $\zb'\in B'$ that is not a local minimum of $F$ there exists a symmetric positive definite matrix $C(\zb')$ such that $(\fb(\zb),C(\zb')(\zb-\zb'))\le 0$ for any $\zb\in U(\zb')$, where $U(\zb')$ is some open neighborhood of $\zb'$.
\end{assumption}
Note that if the second derivatives of the function $F$ exist, the assumption above holds for any critical point of $F$ that is a local maximum. Indeed, in this case 
\begin{align}\label{eq:hess}
 \fb(\zb) = H\{F(\zb')\}(\zb - \zb')+\delta(\|\zb - \zb'\|),
\end{align}
where $\delta(\|\zb - \zb'\|) = o(1)$ as $\zb \to \zb'$ and $H\{F(\cdot)\}$ is the Hessian matrix of $F$ at the corresponding point.

Note that Assumption~\ref{assum:A5A} does not require existing of second derivatives. For example consider the function $F(z),$ $z\in \R$ that behaves in a neighborhood of its local maximum $z=0$ as follows:
$$F(z)=\begin{cases}
         -z^2, \mbox{ if $z\le 0$},\\
            -3z^2, \mbox{ if $z > 0$}.
         \end{cases}
$$
The function above is not twice differentiable at $z=0$, but $\nabla F(z)z\le 0$ for any $z\in\R$.

\begin{theorem}\label{th:th6}
Let the objective function $F(\zb)$ and gradients $\fb_i(\zb)$, $i\in [n]$, in the distributed optimization problem~\eqref{eq:DOP} satisfy \beh{Assumptions~\ref{assum:A1}-\ref{assum:A3}, and~\ref{assum:A5A}}.
Let the sequence of the random i.i.d. vectors $\{\Wb_i(t)\}_t$, $i\in[n]$, satisfy Assumption~\ref{assum:A4}. Let $\{a(t)\}$ be a positive and non-increasing step-size sequence such that $a(t)=O\left(\frac{1}{t^{\nu}}\right)$, where $\frac{1}{2}<\nu\le 1$. Then the average state vector $\bar{\zx}(t)$ (defined by~\eqref{eq:runavpp}) and states $\zb_i(t)$ for the distributed optimization problem~\eqref{eq:pps} with $S$-strongly connected graph sequence $\{G(t)\}$ converges almost surely to a point from the set $B^{''}=B\setminus B^{'}$ of the local minima of the function $F(\zb)$ or to the boundary of one of its connected components, for any initial states $\{\zx_i(0)\}_{i\in[n]}$.
\end{theorem}

\beh{
\begin{rem}
 Instead of Assumption~\ref{assum:A5A}, one can assume that $B'$ is defined as the set of points $\zb'\in\R^d$ for which there exists a symmetric positive definite matrix $C(\zb')$ such that $(\fb(\zb),C(\zb')(\zb-\zb'))\le 0$ for any $\zb\in U(\zb')$. In this case, similar to Theorem~\ref{th:th6}, we can show the almost sure convergence of the process in~\eqref{eq:runavpp} to a point from the set of critical points defined by $B^{''}=B\setminus B^{'}$ or to the boundary of one of its connected components.
\end{rem}
}

\subsection{Convergence Rate of the Perturbed Process}\label{subsec:rate}
In this subsection, we present a result on the convergence rate of the procedure~\eqref{eq:runavpp} introduced above. Recall that Theorem~\ref{th:th6} claims the almost sure convergence of this process either to a point from the set of the local minima of the function in the distributed optimization problem~\eqref{eq:DOP} or to the boundary of one of its connected components, given Assumptions~\ref{assum:A1}-\ref{assum:A4}.
To formulate the result on the convergence rate, we need the following assumption on gradients' smoothness.
\begin{assumption}\label{assum:A5}
 $\frac{\partial f^k(\boldsymbol x)}{\partial x^l}$ exists and is bounded for all $k,l=1\dots d,$ where $f^k$ is the $k$-th coordinate of the vector $\fb$.
\end{assumption}

Now we are ready to formulate the following theorem.

\begin{theorem}\label{th:th7}
{Let the objective function $F(\zb)$ have finitely many critical points, i.e. the set $B$ be finite\footnote{This assumption is made to simplify notations in the proof of this theorem. Note that this theorem can be generalized to the case of finitely many connected components in the set of critical points and local minima.}.}
Let the objective function $F(\zb)$, gradients $\fb_i(\zb)$, and $\fb=\frac{1}{n}\sum_{i=1}^n\fb_i$, $i\in [n]$, in the distributed optimization problem~\eqref{eq:DOP} satisfy \beh{Assumptions~\ref{assum:A1}-\ref{assum:A3},~\ref{assum:A5A}-\ref{assum:A5}}.
Let the sequence of the random i.i.d. vectors $\{\Wb_i(t)\}_t$, $i\in[n]$, satisfy Assumption~\ref{assum:A4} and the graph sequence $\{G(t)\}$ be $S$-strongly connected. Then there exists a constant $\alpha>0$ such that for any $0<a\leq \alpha$, the average state $\bar{\zx}(t)$ (defined by~\eqref{eq:runavpp}) and states $\zb_i(t)$ for the process~\eqref{eq:runavpp} with the choice of step-size $a(t)=\frac{a}{t}$ converge to a point in $B^{''}$ (the set of local minima of $F(\zb)$). Moreover, for any $\zx^*\in B^{''}$
\begin{align*}
 \E\{\|\bar{\zx}(t)-\zx^*\|^2\mid \lim_{s\to\infty}\bar{\zx}(s)=\zx^*\}=O\left(\frac{1}{t}\right) \mbox{ as }t\to\infty.
\end{align*}
\end{theorem}

\section{Proof of the Main Results}\label{sec:proofs}
The rest of the paper is organized to prove results described in Section~\ref{sec:mainresults}.

{
Firstly, note that under  Assumption~\ref{assum:A2} each gradient function $\fb_i(\zx)$, $i\in [n]$, is a Lipschitz function with some constant $l_i$. This allows us to formulate the following lemma which we will need to prove Theorem~\ref{th:th3} and Theorem~\ref{th:th4}.

\begin{lem}\label{lem:lem1}
Let $\{a(t)\}$ be a non-increasing sequence such that $\sum_{t=0}^{\infty}a^2(t)<\infty$. Then, there exists $c(t)$ such that the following holds for the process~\eqref{eq:runavdp1}, given Assumptions~\ref{assum:A1},~\ref{assum:A2} and a $S$-strongly connected graph sequence $\{G(t)\}$:
 \begin{align*}
 \|\zq(t,\bar{\zx}(t))\|&= \|\frac{1}{n}\sum_{i=1}^{n}\fb_i({\zb_i(t+1)})-\fb(\bar{\zx}(t))\|\le c(t),
 \end{align*}
 and
 \begin{align*}
 \sum_{t=0}^{\infty} a(t+1)c(t)<\infty.
 \end{align*}
\end{lem}

\begin{proof}
Since the functions $\fb_i$ are Lipschitz (Assumption~\ref{assum:A2}) and taking into account Theorem~\ref{th:th2}, we have that
\begin{align*}
\|\sum_{i=1}^{n}&\fb_i({\zb_i(t+1)})-\fb(\bar{\zx}(t))\|\le\sum_{i=1}^{n}\|\fb_i({\zb_i(t+1)})-\fb_i(\bar{\zx}(t))\|\cr
                                                              &\le \sum_{i=1}^{n} l_i\|\zb_i(t+1)-\bar{\zx}(t)\|\cr
                                                              &\le \frac{8ln}{\delta}\left(\lambda^t\sum_{j=1}^{n}\|\zx_j(0)\|_1+\sum_{s=1}^{t}\lambda^{t-s}\sum_{j=1}^{n}\|\ze_j(s)\|_1\right)
\end{align*}
for some positive $\delta$ and $\lambda$ and where $l=\max_{i\in[n]}{l_i}$, $\|\ze_i(t)\|=a(t)\|\fb_i(\zb_i(t))\|$.

Let 
$$c(t)=\frac{8l}{\delta}\left(\lambda^t\sum_{j=1}^{n}\|\zx_j(0)\|_1+\sum_{s=1}^{t}\lambda^{t-s}\sum_{j=1}^{n}\|\ze_j(s)\|_1\right).$$ 
Then,
$\|\zq(t,\bar{\zx}(t))\|\le c(t)$. Taking into account Assumption~\ref{assum:A1}, we get 
\begin{align*}
\sum_{t=1}^{\infty}a(t)\|\ze_i(t)\| = \sum_{t=0}^{\infty}a^2(t)\|\fb_i(\zb_i(t))\|\le \alpha\sum_{t=0}^{\infty}a^2(t)<\infty,
\end{align*}
Hence, according to Theorem~\ref{th:th2},
\begin{align*}
&\sum_{t=0}^{\infty}a(t+1)c(t)<\infty.
\end{align*}
\end{proof}
}
\subsection{Proof of Theorem~\ref{th:th3}}
\begin{proof}
To prove this statement we will use the general result formulated in Theorem~\ref{th:th1}. We emphasize one more time that the process~\eqref{eq:runavdp1} under consideration represents a special deterministic case of the recursive procedure~\eqref{eq:basic}. Note that, according to the choice of $a(t)$, Assumption~\ref{assum:A1}, and Lemma~\ref{lem:lem1}, the conditions~\ref{th1:cond3} and~\ref{th1:cond4} of Theorem~\ref{th:th1} hold. Thus, it suffices to show that there exists a sample function $V(t, \zx)$ of the process~\eqref{eq:runavdp1} satisfying the conditions~\ref{th1:cond1} and~\ref{th1:cond2} in Theorem~\ref{th:th1}. A natural choice for such a function is the (time-invariant) function $V(t,\zx)=V(\zx)=F(\zx)+C$, where $C$ is a constant chosen to guarantee the positiveness of $V(\zx)$ over the whole $\mathbb{R}^d$. Note that such constant always exists because of the continuity of $F$ and Assumption~\ref{assum:A3}.
Then, $V(\zx)$ is nonnegative and $V(\zx)\to\infty \quad \mbox{as} \quad \|\zx\|\to\infty$. Let $L V(\zx(t)):=LV(t,\zx)$. 
Now we show that the function $V(\zx)$ satisfies the condition~\ref{th1:cond2} in Theorem~\ref{th:th1}. 
Using the notation $\tilde{\fb}(t,\bar{\zx}(t)) = \fb(\bar{\zx}(t))+\zq(t, \bar{\zx}(t))$, $\zq(t,\bar{\zx}(t))= \frac{1}{n}\sum_{i=1}^{n}\fb_i({\zb_i(t+1)})-\fb(\bar{\zx}(t))$ and by the Mean-value Theorem:
\begin{align*}
LV(\bar{\zx}(t))& = V(\bar{\zx}(t+1)) - V(\bar{\zx}(t)) \cr
&= V(\bar{\zx}(t) - a(t+1)\tilde{\fb}(t,\bar{\zx}(t))) - V(\bar{\zx}(t))\cr
& = -a(t+1)(\nabla V(\tilde{\zx}),\tilde{\fb}(t,\bar{\zx}(t)))\cr
& = -a(t+1)(\nabla V(\bar{\zx}(t)),\tilde{\fb}(t,\bar{\zx}(t)))\cr
&\qquad-a(t+1)(\nabla V(\tilde{\zx})-\nabla V(\bar{\zx}(t)),\tilde{\fb}(t,\bar{\zx}(t))),
\end{align*}
where $\tilde{\zx}=\bar{\zx}(t)-\theta a(t+1)\tilde{\fb}(t,\bar{\zx}(t))$ for some $\theta\in (0,1)$. Taking into account Assumption~\ref{assum:A2}, we
obtain that for some constant $k>0$,
\begin{align*}
 \|\nabla V(\tilde{\zx})-\nabla V(\bar{\zx}(t))\|&\le k\|\tilde{\zx}-\bar{\zx}(t)\|\cr
 &=k a(t+1) \theta \|\tilde{\fb}(t,\bar{\zx}(t))\|.
\end{align*}

Hence, due to the Cauchy-Schwarz inequality,
\begin{align*}
LV(\bar{\zx}(t))
& \le -a(t+1)(\nabla V(\bar{\zx}(t)),\tilde{\fb}(t,\bar{\zx}(t)))\cr
&\qquad +k_1 a^2(t+1)\|\tilde{\fb}(t,\bar{\zx}(t))\|^2\cr
& \le -a(t+1)(\nabla V(\bar{\zx}(t)),\fb(\bar{\zx}(t)))\cr
&\qquad +k_2 a(t+1)\|\zq(t, \bar{\zx}(t))\| \cr
&\qquad+ k_3a^2(t+1)(1+\|\zq(t, \bar{\zx}(t))\|+\|\zq(t, \bar{\zx}(t))\|^2)
\end{align*}
for some $k_1, k_2, k_3>0$.

Recall that the function $V(\zx)$ is nonnegative. Thus, we finally obtain that
\begin{align*}
LV(\bar{\zx}(t))\le -a(t+1)(\nabla V(\bar{\zx}(t)),\fb(\bar{\zx}(t)))+ g(t)(1+V(\bar{\zx}(t))),
\end{align*}
where
\begin{align*}
g(t) =& k_4 (a(t+1)\|\zq(t, \bar{\zx}(t))\|+a^2(t+1)\cr
& +a^2(t+1)\|\zq(t, \bar{\zx}(t))\|
+a^2(t+1)\|\zq(t, \bar{\zx}(t))\|^2)
\end{align*}
for some constant $k_4>0$. Lemma~\ref{lem:lem1} and the choice of $a(t)$ imply that $g(t)>0$ and
$\sum_{t=0}^{\infty}g(t)<\infty$.
Hence,
\begin{align*}
LV(\bar{\zx}(t))\le -a(t+1)\phi(t,\bar{\zx}(t)) + g(t)(1+V(\bar{\zx}(t))),
\end{align*}
where, according to the choice of $V(\zx)$,
\begin{align*}
\phi(t,\bar{\zx}(t))=(\nabla V(\bar{\zx}(t)),\fb(\bar{\zx}(t)))=\|\fb(\bar{\zx}(t))\|^2\in\Phi(B).
\end{align*}
Thus, all conditions of Theorem~\ref{th:th1} hold and we conclude that either $\lim_{t\to\infty} \bar{\zx}(t)=\zb_0\in B$ or $ \bar{\zx}(t)$ converges to the boundary of a connected component of the set $B$. By Theorem~\ref{th:th2}, $\lim_{t\to\infty}\|\zb_i(t)-\bar{\zx}(t)\|=0$ for all $i\in [n]$ and, hence, the result follows.
\end{proof}
\smallskip

Unfortunately (and naturally), there exists no function $V(t, \zx)$ for which
\begin{align*}
\phi(t,\bar{\zx}(t))=(\nabla V(\bar{\zx}(t)),\fb(\bar{\zx}(t)))\in\Phi(B^{''}),
\end{align*}
 where $B^{''}$ represents the set of local minima of the function $F(\zb)$. Thus, the deterministic process~\eqref{eq:runavdp1} is unable to distinguish between local minima, saddle-points, and local maxima and guarantees only the convergence to a zero point of the gradient $\nabla F$.
Theorem~\ref{th:th4} and Theorem~\ref{th:th6} provide a solution on how to rectify this issue.

\subsection{Proof of Theorem~\ref{th:th4}}
\begin{proof}Assumptions~\ref{assum:A1},~\ref{assum:A2}, and~\ref{assum:A4} allow us to use Theorem~\ref{th:th2} to conclude that all $\zb_i(t+1)$ in~\eqref{eq:pps} converge to the average state $\bar{\zx}(t)$ almost surely, if $a(t)\to 0$ as $t\to\infty$.
Moreover, if $\sum_{t=0}^{\infty}a^2(t)<\infty$, Lemma~\ref{lem:lem1} holds for the process~\eqref{eq:runavpp}, where $\|\zq(t,\bar{\zx}(t))\|\le c'(t)$ a.s., $c'(t)=M \left(\lambda^t + \sum_{s=1}^{t}\lambda^{t-s}a(s)\right)$ for some positive $M$, and $\sum_{t=0}^{\infty}a(t+1)c'(t)<\infty$. Now, considering $V(\zx)=F(\zx)+C$, where $C$ is such that $V(\zx)>0$ for all $\zx$, and noticing that $LV(\bar{\zx}(t))=\E[V(\{\bar{\zx}(t+1)|\bar{\zx}(t)=\zx\})-V(\zx)]$, we can use the Mean-value Theorem for the term under the expectation (as in the proof of Theorem~\ref{th:th3}), and verify that all conditions of Theorem~\ref{th:th1} hold for the process~\eqref{eq:runavpp}, given that Assumption~\ref{assum:A3} holds and the step-size $\{a(t)\}$ is appropriately chosen.
\end{proof}

\subsection{Proof of Theorem~\ref{th:th6}}
 To prove Theorem~\ref{th:th6} we use the following result from~\cite{NH}, Chapter 5.
 We start by formulating the following important lemma (Lemma 5.4.1,~\cite{NH}) for a general Markov process $\{x(t)\}_t$ defined on some state space $E$.

\begin{lem}\label{lem:lem2}
Let $\tilde{x}$ be a point in $E$, and $D$ some bounded \beh{open domain} containing $\tilde{x}$. Assume there exists a function $V(t,x)$, bounded below for $t\in \Z^+$ and $x\in D$, such that for the process $\{x(t)\}_t$ the following conditions hold:
\begin{enumerate}[(1)]
\item \label{item:lema21}$LV(t,x)\le \gamma(t)$ for $t\in \Z^+$ and $x\in D$, where
$\sum_{t=1}^{\infty}\gamma(t)<\infty$,
\item \label{item:lema22} $\lim_{t\to\infty}V(t,x(t))=\infty$
for any sequence of points $x(t)\in E$ such that $\lim_{t\to\infty}x(t)=\tilde{x}$.
\end{enumerate}
Then $\Pr\{\lim_{t\to\infty}x(t)=\tilde{x}\}=0$ independently on the initial state $x(0)$.
\end{lem}

This lemma is used to prove the following theorem that is a special case of Theorem 5.4.1 in~\cite{NH}.
\smallskip

\begin{theorem}\label{th:th5}
Let $\{x(t)\}_t$ be a Markov process on $\mathbb{R}^d$ defined by:
\begin{align*}
\zx(t+1)=\zx(t) - a(t+1) &[\fb(\zx(t))+\zq(t,x(t))] \cr
&\qquad\qquad- a(t+1)\Wb(t+1),
\end{align*}
where $\{\Wb(t)\}_t$ is a sequence of i.i.d. random vectors.
Let $H'$ be the set of the points $\zb'\in\mathbb{R}^d$ for which there exists a symmetric positive definite matrix $C=C(\zb')$  such that $(\fb(\zb),C(\zb-\zb'))\le 0$  \beh{for $\zb\in U(\zb')$, where $U(\zb')$ is some open neighborhood of $\zb'$}.
Assume that
\begin{enumerate}[(1)]
 \item\label{th5:cond1} for any $\zb^{'}\in H^{'}$ there exists positive constants $\delta=\delta(\zb')$ and $K=K(\zb')$ such that
 $\|\fb(\zb)\|\le K\|\zb - \zb'\|$ for any $\zb:$ $\|\zb - \zb'\|<\delta$,
 \item\label{th5:cond2} the random vectors $\Wb(t)$ satisfy Assumption~\ref{assum:A4},
 \item\label{th5:cond3} $\sum_{t=0}^{\infty}a^2(t)<\infty$, $\sum_{t=0}^{\infty}\left(\frac{a(t)}{\sqrt{\sum_{k=t+1}^{\infty}a^2(k)}}\right)^3<\infty$,
 \item\label{th5:cond4} $\sum_{t=0}^{\infty}\frac{a(t)q(t)}{\sqrt{\sum_{k=t+1}^{\infty}a^2(k)}}<\infty$, $q(t)=\sup_{\zx\in\mathbb{R}^d}\|\zq(t,\zx)\|$.
\end{enumerate}
Then $\Pr\{\lim_{t\to\infty}\zx(t)\in H'\}=0$ irrespective of the initial state $x(0)$.
\end{theorem}
\begin{proof}
We provide here a sketch of the proof. All details can be found in~\cite{NH}.

Let $\zx'\in H'$ be any point from the set $H'$ and $C=C(\zx')$ be the matrix figuring in the definition of $H'$. Without loss of generality we may assume $\zx'=\zz$. Let
\begin{align*}
U(\zx)& =(C\zx, \zx),\cr
W(y)& = \int_{0}^{y}dv\int_{0}^{v}\frac{e^{u-v}}{\sqrt{uv}}du, \mbox{ } y\ge0, \\
\phi(t)& =2\Tr[C]\sum_{u=t}^{\infty}a^2(u),\cr
T(t)&= -\ln\phi(t), \cr
 V(t,\zx)&=T(t)-W(y),
\end{align*}
where $y=\frac{U(\zx)}{\phi(t)}$.
Then expanding $LV(t,\zx)$ at the point $y=\frac{U(\zx)}{\phi(t)}$ by Taylor's formula and using the analytical properties of the function $W(y)$ and assumptions~\ref{th5:cond1} and~\ref{th5:cond2} of the theorem, for some positive constant $k$ we have
\begin{align*}
&LV(t,\zx)\cr
&\le k \left[a^2(t)+\frac{a(t)q(t)}{\sqrt{\sum_{k=t+1}^{\infty}a^2(k)}}+\left(\frac{a(t)}{\sqrt{\sum_{k=t+1}^{\infty}a^2(k)}}\right)^3\right],
\end{align*}
for any \beh{$\zx\in D = \{\|\zx\|<\delta\}\cap U(\zz)$, where $U(\zz)$ and the constant $\delta$ are used in the definition of the set $H'$ and condition~\eqref{th5:cond1} of theorem, respectively}.

Let
\[\gamma(t):= \left[a^2(t)+\frac{a(t)q(t)}{\sqrt{\sum_{k=t+1}^{\infty}a^2(k)}}+\left(\frac{a(t)}{\sqrt{\sum_{k=t+1}^{\infty}a^2(k)}}\right)^3\right].\]
Using assumption~\ref{th5:cond3} of the theorem, we have that $\sum_{t=0}^{\infty}\gamma(t)<\infty$. Thus, condition~\ref{item:lema21} of Lemma~\ref{lem:lem2} is fulfilled for the introduced function $V(t,\zx)$ in $D$.
Now we show that $V(t,\zx)$ and the point $\zx'$ also satisfy condition~\ref{item:lema22} of Lemma~\ref{lem:lem2}. Let $\{\zx(t)\}_t$ be any sequence in $\mathbb R^d$ such that $\zx(t)\to\zz$ as $t\to\infty$.
One can check that there exists a positive constant $c$ such that $W(y)<\ln(1+y)+c$ for all $y\ge 0$ and, hence\footnote{We refer the readers to Lemma 5.3.3 in~\cite{NH} for more details about the analytical properties of $W(y)$.},
\begin{align*}
V(t,\zx)& =T(t)-W\left(\frac{U(\zx(t))}{\phi(t)}\right)\cr
		&\ge -\ln\phi(t)-\ln\left(1+\frac{U(x(t))}{\phi(t)}\right)-c\\
		& =-\ln(\phi(t)+U(\zx(t)))-c\to\infty, \quad \mbox{as } t\to\infty,
\end{align*}
since $\phi(t)\to 0$ and $U(x(t))\to U(\zz)=0$ as $t\to\infty$.
Thus, condition 2 of Lemma~\ref{lem:lem2} is also fulfilled. Moreover, it is readily seen that $V(t,x)$ is bounded below for all $t\in \Z^+$ and $\zx\in D$. Thus, the function $V(t,\zx)$ and any point $\zx'\in H'$ satisfy Lemma~\ref{lem:lem2}. Then we can conclude that $\Pr\{\lim_{t\to\infty}\zx(t)\in H'\}=0$.
\end{proof}

With this, we can prove Theorem~\ref{th:th6}.

\begin{proof}[Proof of Theorem~\ref{th:th6}]
We first notice that, under the proposed choice of $a(t)$, $\sum_{t=0}^{\infty}a(t)=\infty$ and $\sum_{t=0}^{\infty}a^2(t)<\infty$. Hence, we can use Theorem~\ref{th:th4} to conclude that the process~\eqref{eq:runavpp} converges to a point from the set $B$ represented by critical points of the function $F(\zb)$ (or to the boundary of one of connected components of $B$). As before, let the set of points that are not local minima be denoted by $B'$, $B'=B\setminus B^{''}$.
Next, we notice that, \beh{according to Assumption~\ref{assum:A5A}, for any $\zb'\in B'$ there exist some symmetric positive definite matrix $C=C(\zb')$ such that $(\fb(\zb),C(\zb-\zb'))\le 0$ for $\zb\in U(\zb')$, where $U(\zb')$ is some neighborhood of $\zb'$.} Moreover, since $B'\subset B$ and due to Assumption~\ref{assum:A2}, we can conclude that the condition~\ref{th5:cond1} of Theorem~\ref{th:th5} holds for $B'$ and $\fb$.

It is straightforward to verify that the sequence $a(t)=O\left(\frac{1}{t^{\nu}}\right)$ satisfies condition~\ref{th5:cond3} of Theorem~\ref{th:th5}.  Next, recall (see the discussion in the proof of Theorem~\ref{th:th4}) that
\begin{align*}
\|\zq(t,\bar{\zx}(t))\|&\le \frac1n\sum_{i=1}^{n} l_i\|\zb_i(t+1)-\bar{\zx}(t)\|\cr
&\le M \left(\lambda^t + \sum_{s=1}^{t}\lambda^{t-s}a(s)\right),
\end{align*}
 almost surely for some positive constant $M$. Hence,
\begin{align*}
 q(t)=\sup_{\zx\in\mathbb{R}^d}\|\zq(t,\zx)\|\le M \left(\lambda^t + \sum_{s=1}^{t}\lambda^{t-s}a(s)\right).
\end{align*}
Taking into account this and the fact that $\frac{a(t)}{\sqrt{\sum_{k=t+1}^{\infty}a^2(k)}}=O(\frac{1}{\sqrt{t}})$, we have
\begin{align*}
 &\sum_{t=0}^{\infty}\frac{a(t)q(t)}{\sqrt{\sum_{k=t+1}^{\infty}a^2(k)}}\cr
 &\qquad\le \sum_{t=0}^{\infty} O\left(\frac{1}{\sqrt {t}}\right) \left(\lambda^t + \sum_{s=1}^{t}\lambda^{t-s}O\left(\frac{1}{s^{\nu}}\right)\right)<\infty.
\end{align*}
The last inequality is due to following considerations.
\begin{align*}
 1)& \quad \sum_{t=0}^{\infty}O\left(\frac{1}{\sqrt {t}}\right)\lambda^t\le \sum_{t=0}^{\infty} O(\lambda^t)<\infty,\mbox{ since }\lambda\in(0,1).\\
 2)& \quad\sum_{t=0}^{\infty}O\left(\frac{1}{\sqrt {t}}\right)\left(\sum_{s=1}^{t}\lambda^{t-s}O\left(\frac{1}{s^{\nu}}\right)\right)\cr
 &\qquad\qquad\le \sum_{t=0}^{\infty}\sum_{s=1}^{t}\lambda^{t-s}O\left(\frac{1}{s^{\nu+1/2}}\right)<\infty,
\end{align*}
since, according to~\cite{25A1}, any series of the type $\sum_{k=0}^{\infty}\left(\sum_{l=0}^{k}\beta^{k-l}\gamma_l\right)$ converges, if $\gamma_k\ge 0$ for all $k\ge0$, $\sum_{k=0}^{\infty}\gamma_k<\infty$, and $\beta\in(0,1)$.

Thus, all conditions of Theorem~\ref{th:th5} are fulfilled for the points from the set $B'$. It implies that the process~\eqref{eq:runavpp} cannot converge to the points from the set $B^{'}$ and, thus, either $\Pr\{\lim_{t\to\infty}\bar{\zx}(t)=\zb_0\in B^{''}\}=1$ or $\bar{\zx}(t)$ converges almost surely to the boundary of one of connected components of the set $B^{''}$. We conclude the proof by noting that, according to Theorem~\ref{th:th2}, $\lim_{t\to\infty}\|\zb_i(t)-\bar{\zx}(t)\|=0$ almost surely for all $i\in [n]$ and, hence, the result follows.
\end{proof}

\subsection{Proof of Theorem~\ref{th:th7}}
We start by revisiting a well-known result in linear systems theory~\cite{malkin}.

\begin{lem}\label{lem:lem3}
If some matrix $A$ is stable\footnote{A matrix is called stable (Hurwitz) if all its eigenvalue have a strictly negative real part.}, then for any symmetric positive definite matrix $D$, there exists a symmetric positive definite matrix $C$ such that $CA+A^TC=-D$. In fact, $C=\int_0^{\infty}e^{A^T\tau}De^{A\tau}d\tau$.
\end{lem}

Recall that we deal with the objective function $F(\zb)$ that has finitely many local minima $\{\zx^*_1, \ldots, \zx^*_R\}$. We begin by noticing that according to Assumption~\ref{assum:A5} the function $\fb(\zx)$ admits the following representation in some neighborhood of any local minimum $\zx^*_m$, $m=1,\ldots, R$,
\begin{align*}
 \fb(\zx) = H\{F(\zx^*_m)\}(\zx - \zx^*_m)+\delta_m(\|\zx - \zx^*_m\|),
\end{align*}
where $\delta_m(\|\zx - \zx^*_m\|) = o(1)$ as $\zx \to \zx^*_m$ and $H\{F(\cdot)\}$ is the Hessian matrix of $F$ at the corresponding point.

For the sake of notational simplicity, let the matrix $H\{F(\zx^*_m)\}$ be denoted by $H_m$. Now we are ready to formulate the following lemma.
\begin{lem}\label{lem:lem4}
Under Assumption~\ref{assum:A5}, for any local minimum $\zx^*_m$, $m=1,\ldots, R$, of the objective function $F$ there exist a symmetric positive definite matrix $C_m$ and positive constants $\beta(m)$, $\varepsilon(m)$, and $a <\infty$ (independent on $m$), such that for any $\zx$: $\|\zx-\zx^*_m\|<\varepsilon(m)$ the following holds:
\begin{align*}
(C_m\fb(\zx),\zx-\zx^*_m)& \ge \beta(m)(C_m(\zx-\zx^*_m),\zx-\zx^*_m),\\
2a\beta(m)& >1.
\end{align*}
\end{lem}

\begin{proof}
Since $\zx^*_m$, for $m\in [R]$, is a local minima of $F$, we can conclude that $H_m$, $m\in [R]$, is a symmetric positive definite matrix. Hence, there exists a finite constant $a>0$ such that the matrix $-aH_m+\frac12 I$ is stable for all $m\in [R]$, where $I$ is the identity matrix.

Without loss of generality we assume that $\zx^*_m=\zz$.
Let $\lambda_1(m),\ldots,\lambda_d(m)$ be the eigenvalues of $H_m$, $\tilde{\lambda}(m)=\min_{j\in [d]} \lambda_j(m)$.
Since the matrix $-aH_m+\frac12 I$ is stable, we conclude that $2a\tilde{\lambda}(m)>1$.
Moreover, $-H_m+\lambda_0(m) I$ is stable as well for any $\lambda_0(m)<\tilde{\lambda}(m)$, since the matrix $-H_m$ is stable. Hence, we can apply Lemma~\ref{lem:lem3} to $-H_m+\lambda_0(m) I$ which implies that there exists a matrix $C_m=C_m(\lambda_0)$ such that
\begin{align}\label{eq:eq2}
(C_mH_m\zx,\zx)\ge\lambda_0(m)(C_m\zx,\zx).
\end{align}

Now we remind that in some neighborhood of $\zz$
\begin{align*}
 \fb(\zx) = H_m\zx +\delta_m(\|\zx \|), \quad \delta_m(\|\zx \|) = o(1) \mbox{ as } \zx \to \zz.
\end{align*}
Thus, taking into account~\eqref{eq:eq2}, we conclude that for any $\beta(m)<\lambda_0(m)$ there exists $\varepsilon(m)>0$ such that for any $\zx$ with $\|\zx\|<\varepsilon(m)$,  $
 (C_m\fb(\zx),\zx)\ge\beta(m)(C_m\zx,\zx)$. As the constants $\lambda_0(m)<\tilde{\lambda}(m)$ and $\beta(m)<\lambda_0(m)$ can be chosen arbitrarily and $2a\tilde{\lambda}(m)>1$, we conclude that
$2a{\beta}(m)>1$. That completes the proof.
\end{proof}

\begin{proof}[Proof of Theorem~\ref{th:th7}]
We proceed to show this result in two steps:  we first show that a trimmed version of the dynamics converges on $O\left(\frac{1}{t}\right)$ and then, we show that the convergence rate of the trimmed dynamics and the original dynamics are the same. Without loss of generality we assume that $\zx^*=\zz$ and
\[\Pr\{\lim_{s\to\infty}\bar{\zx}(s)=\zx^*=\zz\}>0.\]
From Lemma~\ref{lem:lem4}, we know that there exist a symmetric positive definite matrix $C$ and positive constants $\beta$, $\varepsilon$, and $a$ such that $(C\fb(\zx),\zx)\ge \beta (C\zx,\zx),$ for any $\zx$: $\|\zx\|<\varepsilon$. Moreover, $2a\beta >1$.

Let us consider the following trimmed process $\zxk(t)$ defined by:
\begin{align*}
\hat{\zx}^{\tau,\boldsymbol{\kappa}}(t+1)&=\hat{\zx}^{\tau,\boldsymbol{\kappa}}(t)-\frac{a}{t+1}\left(\zhf(\hat{\zx}^{\tau,\boldsymbol{\kappa}}(t))+\zq(t,\bar{\zx}(t))\right)\cr
&\qquad-\frac{a}{t+1}\hat{\Wb}(t+1,\hat{\zx}^{\tau,\boldsymbol{\kappa}}(t)),
\end{align*}
for $t \ge\tau$ with $\zxk(\tau)=\boldsymbol{\kappa}$, where
the random vector $\bar{\zx}(t)$ is updated according to~\eqref{eq:runavpp} with $a(t)=\frac{a}{t}$,
\begin{align*}
\zhf(\zx)& =\begin{cases} \fb(\zx), &\mbox{if } \|\zx\|<\varepsilon \\
\beta \zx, & \mbox{if } \|\zx\|\ge \varepsilon
\end{cases},\\
\hat{\Wb}(t,{\zx})& =\begin{cases} \Wb(t), &\mbox{if } \|\zx\|<\varepsilon \\
\zz, & \mbox{if } \|\zx\|\ge \varepsilon
\end{cases}.
\end{align*}

Next we show the $O\left(\frac{1}{t}\right)$ rate of convergence for the trimmed process. The proof follows similar argument as of Lemma~6.2.1 of~\cite{NH}. Obviously, $(C\zhf(\zx),\zx)\ge \beta (C\zx,\zx),$ for any $\zx\in\mathbb{R}^d$.
We proceed by showing that $\E\|\hat{\zx}^{\tau,\boldsymbol{\kappa}}(t)\|^2=O\left(\frac{1}{t}\right)$ as $t\to\infty$ for any $\tau\ge 0$ and $\boldsymbol{\kappa}\in\mathbb{R}^d$. For this purpose we consider the function $V_1(\zx) = (C\zx,\zx)$. Applying the generating operator $L$ of the process $\zxk(t)$ to this function, we get that
\begin{align*}
&LV_1(\zxk) =\cr
&\E\left[C\left(\zx-\frac{a(\zhf({\zx})+\zq(t,\bar{\zx}(t)))+\hat{\Wb}(t+1,\bar{\zx}(t))}{t+1}\right),\right.\cr
&\quad\left.\zx-\frac{a(\zhf({\zx})+\zq(t,\bar{\zx}(t)))+\hat{\Wb}(t+1,\bar{\zx}(t))}{t+1}\right]-(C\zx,\zx)\\
& =-\frac{2a}{t+1}\E(C\zx,\zhf(\zx)+\zq(t,\bar{\zx}(t))+\hat{\Wb}(t+1,\bar{\zx}(t)))\\
&+\frac{a^2\E(C(\zhf(\zx)+\zq(t,\bar{\zx}(t))),\zhf(\zx)+\zq(t,\bar{\zx}(t)))}{(t+1)^2}\\
&+\frac{a^2\E\hat{\Wb}^2(t+1,\bar{\zx})}{(t+1)^2}\\
& \le -\frac{2a}{t+1}(C\zhf(\zx),\zx)+\frac{2a}{t+1}\E|(C{\zq}(t,\bar{\zx}(t)),\zx)|\\
& +\frac{a^2\|C\|\E(\|\zhf(\zx)+\zq(t,\bar{\zx}(t))\|^2+1)}{(t+1)^2}.
\end{align*}
Recall from the proof of Lemma~\ref{lem:lem1} that for some positive constants $M$ and $M'$ almost surely
\begin{align}\label{eq:eq3}
\|\zq(t,\bar{\zx}(t))\|\le c(t) = M\left(\lambda^t+\frac{a}{t+1}\sum_{s=1}^{t}\lambda^{t-s}\right)\le M',
\end{align}
for $t\in \Z^+$.
According to the definition of the function $\zhf$ and Assumption~\ref{assum:A1} correspondingly, $\|\zhf(\zx)\|=\beta\|\zx\|$ for $\|\zx\|\ge\varepsilon$ and $\zhf(\zx)$ is bounded for $\|\zx\|<\varepsilon$. Thus, taking into account~\eqref{eq:eq3}, we conclude that there exists a positive constant $k$ such that  $\E(\|\zhf(x)+\zq(t,\bar{\zx}(t))\|^2+1)\le k(\|\zx\|^2+1)$. Hence, using the Cauchy-Schwarz inequality and the fact that $\|\zx\|\le 1+\|\zx\|^2$, we obtain
\begin{align*}
LV_1(\zxk)& \le -\frac{2a}{t+1}(C\zhf(\zx),\zx)+\frac{2ac(t)\|C\|(1+\|\zx\|^2)}{t+1}\cr
&\qquad+
\frac{k_2(1+\|\zx\|^2)}{(t+1)^2}\\
& \le -\frac{2a}{t+1}(C\zhf(\zx),\zx)+\frac{k_1c(t)(1+(C\zx,\zx))}{t+1}\cr
&\qquad+
\frac{k_3(1+(C\zx,\zx))}{(t+1)^2},
\end{align*}
where $k_1,$ $k_2$, and $k_3$ are some positive constants. Taking into account that $(C\zhf(\zx),\zx)\ge \beta (C\zx,\zx)$ and $2a\beta >1$, we conclude that
\begin{align*}
 LV_1(\zxk)&\le -\frac{p_1V_1(\zxk)}{t+1}\\
 &+(1+V_1(\zxk))\left(\frac{k_1c(t)}{t+1}+\frac{k_3}{(t+1)^2}\right),
\end{align*}
where $p_1>1$.
According to~\eqref{eq:eq3}, there exists such constant $p\in(1, p_1)$ that
\begin{align*}
 -\frac{p_1}{t+1}+\frac{k_1c(t)}{t+1}\ge-\frac{p}{t+1}, \quad -\frac{p_1}{t+1}+\frac{k_3}{(t+1)^2}\ge-\frac{p}{t+1}
\end{align*}
are fulfilled simultaneously, if $t>T$, where $T\ge\tau$ is some finite sufficiently large constant. Thus, for some $p>1$ and $T\ge 0$, we have
\begin{align*}
 LV_1(\zxk)\le -\frac{pV_1(\zx)}{t+1}+\frac{k_1c(t)}{t+1}+\frac{k_3}{(t+1)^2}, \mbox{ for any } t> T.
\end{align*}
Therefore, for $t>T$
\begin{align*}
 &\E V_1(\zxk(t+1))-\E V_1(\zxk(t))\cr
 &\qquad\le -\frac{p\E V_1(\zxk)}{t+1}+\frac{k_1c(t)}{t+1}+\frac{k_3}{(t+1)^2}.
\end{align*}
Using the above inequality, we have
\begin{align*}
 \E V_1&(\zxk(t+1)) \le\E V_1(\zxk(T))\prod_{r=T}^{t}\left(1-\frac{p}{r+1}\right)\cr
 & +k_1\sum_{r=T}^{t}\frac{1}{(r+1)^2}\prod_{m=r+1}^{t}\left(1-\frac{p}{m+1}\right)\\
& +k_3\sum_{r=T}^{t}\frac{c(r)}{r+1}\prod_{m=r+1}^{t}\left(1-\frac{p}{m+1}\right).
\end{align*}
Since
\begin{align*}
 &\prod_{m=r+1}^{t}\left(1-\frac{1}{m+1}\right) \le\exp{\left(-\sum_{m=r+1}^{t}\frac{1}{m+1}\right)},\\
 &\sum_{m=r+1}^{t}\frac{1}{m+1} \ge\int_{r+1}^{t}\frac{1}{m+1}dm = \frac{\ln(t+1)}{\ln (r+1)},
\end{align*}
we get that for some $k_4>0$
\begin{align*}
 \prod_{m=r+1}^{t}\left(1-\frac{p}{m}\right)\le k_4\left(\frac{r+1}{t+1}\right)^p,
\end{align*}
and, hence, for some $k_5, k_6>0$
\begin{align*}
 \sum_{r=T}^{t}\frac{1}{(r+1)^2}\prod_{m=r+1}^{t}\left(1-\frac{p}{m+1}\right)&\le \frac{k_5}{t+1},\\
 \sum_{r=T}^{t}\frac{c(r)}{r+1}\prod_{m=r+1}^{t}\left(1-\frac{p}{m+1}\right)&\le \frac{k_6}{t+1}.
\end{align*}
The last inequalities are due to~\eqref{eq:eq3} and the fact that
$\sum_{r=1}^{t}r^{p-2}=O(t^{p-1})$\footnote{This estimation can be obtained by considering the sum $\sum_{r=1}^{t}r^{p-2}$ the low sum of the corresponding integrals for two cases: $p\ge2$, $1<p<2$.}. Thus, we finally conclude that
\begin{align*}
 \E V_1(\hat{\zx}^{\tau,\mathbf{\kappa}}(t+1))=O\left(\frac{1}{t^p}\right)+O\left(\frac{1}{t}\right)=O\left(\frac{1}{t}\right),
\end{align*}
that implies
\begin{align}\label{eq:eq4}
\E\|\hat{\zx}^{\tau,\boldsymbol{\kappa}}(t)\|^2=O\left(\frac{1}{t}\right) \mbox{ as $t\to\infty$ for any $\tau\ge 0$, $\boldsymbol{\kappa}\in\mathbb{R}^d$},
\end{align}
since $V_1(\zx)=(C\zx,\zx)$ and the matrix $C$ is positive definite.

Since we assumed that $\Pr\{\lim_{s\to\infty}\bar{\zx}(s)=\zz\}>0$, the rate of convergence of $\E\{\|\bar{\zx}(t)\|^2\mid\lim_{s\to\infty}\bar{\zx}(s)=\zz\}$ and $\E\{\|\bar{\zx}(t)\|^2 \mathds{1}(\lim_{s\to\infty}\bar{\zx}(s)=\zz)\}$ will be the same, where $\mathds{1}\{\cdot\}$ denotes the event indicator function. But
\begin{align*}
&\E(\|\bar{\zx}(t)\|^2\mathds{1}\{\lim_{s\to\infty}\bar{\zx}(s)=\zz\})\cr  &\quad=\int_{\mathbb{R}}x^2d\left(\Pr\{\|\bar{\zx}(t)\|^2\le x, \lim_{s\to\infty}\bar{\zx}(s)=\zz\}\right).
\end{align*}
Thus, to get the asymptotic behavior of $\E\{\|\bar{\zx}(t)\|^2\mid\lim_{s\to\infty}\bar{\zx}(s)=\zz\}$ we need to analyze the asymptotics of the distribution function $\Pr\{\|\bar{\zx}(t)\|^2\le x, \lim_{s\to\infty}\bar{\zx}(s)=\zz\}$. For this purpose we introduce the following events:
\begin{align*}
 \Theta_{u,\tilde{\varepsilon}} = \{\|\bar{\zx}(u)\|<\tilde{\varepsilon}\}, \mbox{ } \Omega_{u,\tilde{\varepsilon}} = \{\|\bar{\zx}(m)\|<\tilde{\varepsilon}, \mbox{ } m\ge u\}.
\end{align*}
Since $\bar{\zx}(t)$ converges to $\zx^*=\zz$ with some positive probability, then for any $\sigma>0$ there exist $\tilde{\varepsilon}(\sigma)$ and $u(\sigma)$ such that for any $u\ge u(\sigma)$ the following is hold
\begin{align*}
  &\Pr\{\{\lim_{s\to\infty}\bar{\zx}(s)=\zz\}\mbox{ }\Delta \mbox{ } \Omega_{u,\tilde{\varepsilon}(\sigma)}\}<\sigma,\cr
  &\Pr\{\{\lim_{s\to\infty}\bar{\zx}(s)=\zz\}\mbox{ }\Delta \mbox{ } \Theta_{u,\tilde{\varepsilon}(\sigma)}\}<\sigma,\\
  &\Pr\{\Theta_{u,\tilde{\varepsilon}(\sigma)}\mbox{ }\Delta \mbox{ } \Omega_{u,\tilde{\varepsilon}(\sigma)}\}<\sigma,
\end{align*}
where $\Delta$ denotes the symmetric difference. Hence, choosing $\varepsilon'=\min(\tilde{\varepsilon},\varepsilon)$, we get that for $u\ge u(\sigma)$
\begin{align*}
&\Pr\{\|\bar{\zx}(t)\|^2\le x, \lim_{s\to\infty}\bar{\zx}(s)=\zz\}&\cr
&\qquad\le
\Pr\{\|\bar{\zx}(t)\|^2\le x, \{\lim_{s\to\infty}\bar{\zx}(s)=\zz\}\cup\{\Omega_{u,\varepsilon'}\}\}\\
&\qquad \le\Pr\{\|\bar{\zx}(t)\|^2\le x, \Omega_{u,{\varepsilon'}}\}+\sigma \cr
&\qquad =\Pr\{\|\hat{\zx}^{u,\bar{\zx}(u)}(t)\|^2\le x, \Omega_{u,{\varepsilon'}}\}+\sigma \\
&\qquad\le \Pr\{\|\hat{\zx}^{u,\bar{\zx}(u)}(t)\|^2\le x, \Theta_{u,{\varepsilon'}}\}+2\sigma.
\end{align*}
Taking into account the Markovian property of the process $\{\hat{\zx}^{u,\bar{\zx}(u)}(t)\}$, we conclude from the inequality above that
\begin{align}\label{eq:eq5}
&\varlimsup_{t\to\infty}\Pr\{\|\bar{\zx}(t)\|^2\le x,\lim_{s\to\infty}\bar{\zx}(s)=\zz\}\\\nonumber
&\qquad\le \varlimsup_{t\to\infty}\Pr\{\|\hat{\zx}^{u,\bar{\zx}(u)}(t)\|^2\le x, \Theta_{u,{\varepsilon'}}\}+2\sigma \nonumber\\
&\qquad= \varlimsup_{t\to\infty}\Pr\{\|\hat{\zx}^{u,\bar{\zx}(u)}(t)\|^2\le x\} \Pr\{\Theta_{u,{\varepsilon'}}\}+2\sigma \nonumber \cr
&\qquad\le \varlimsup_{t\to\infty}\Pr\{\|\hat{\zx}^{u,\bar{\zx}(u)}(t)\|^2\le x\} \Pr\{\lim_{s\to\infty}\bar{\zx}(s)=\zz\}+3\sigma.
\end{align}
Analogously,
\begin{align}
&\varliminf_{t\to\infty}\Pr\{\|\bar{\zx}(t)\|^2\le x, \lim_{s\to\infty}\bar{\zx}(s)=\zz\}\nonumber \\
\label{eq:eq6}
&\qquad\le \varliminf_{t\to\infty}\Pr\{\|\hat{\zx}^{u,\bar{\zx}(u)}(t)\|^2\le x\} \Pr\{\lim_{s\to\infty}\bar{\zx}(s)=\zz\}+3\sigma.
\end{align}
Similarly, we can obtain that
\begin{align*}
&\Pr\{\|\bar{\zx}(t)\|^2\le x, \lim_{s\to\infty}\bar{\zx}(s)=\zz\} \\
&\qquad\ge\Pr\{\|\bar{\zx}(t)\|^2\le x, \Omega_{u,{\varepsilon'}}\setminus\{\Omega_{u,{\varepsilon'}}\cap \{\lim_{s\to\infty}\bar{\zx}(s)=\zz\}\}\} \\
&\qquad\ge \Pr\{\|\hat{\zx}^{u,\bar{\zx}(u)}(t)\|^2\le x, \Theta_{u,{\varepsilon'}}\}-2\sigma,
\end{align*}
and, hence,
\begin{align}
&\varlimsup_{t\to\infty}\Pr\{\|\bar{\zx}(t)\|^2\le x, \lim_{s\to\infty}\bar{\zx}(s)=\zz\}\nonumber \\
\label{eq:eq7}
&\qquad\ge \varlimsup_{t\to\infty}\Pr\{\|\hat{\zx}^{u,\bar{\zx}(u)}(t)\|^2\le x\} \Pr\{\lim_{s\to\infty}\bar{\zx}(s)=\zz\}-3\sigma,
\end{align}
\begin{align}
&\varliminf_{t\to\infty}\Pr\{\|\bar{\zx}(t)\|^2\le x, \lim_{s\to\infty}\bar{\zx}(s)=\zz\}\nonumber \\
\label{eq:eq8}
&\qquad \ge \varliminf_{t\to\infty}\Pr\{\|\hat{\zx}^{u,\bar{\zx}(u)}(t)\|^2\le x\} \Pr\{\lim_{s\to\infty}\bar{\zx}(s)=\zz\}-3\sigma.
\end{align}
Since $\sigma$ can be chosen arbitrary small,~\eqref{eq:eq5}-\eqref{eq:eq8} imply that
\begin{align*}
&\varlimsup_{t\to\infty}\Pr\{\|\bar{\zx}(t)\|^2\le x, \lim_{s\to\infty}\bar{\zx}(s)=\zz\}\\
&\qquad =\varlimsup_{t\to\infty}\Pr\{\|\hat{\zx}^{u,\bar{\zx}(u)}(t)\|^2\le x\} \Pr\{\lim_{s\to\infty}\bar{\zx}(s)=\zz\},
\end{align*}
\begin{align*}
&\varliminf_{t\to\infty}\Pr\{\|\bar{\zx}(t)\|^2\le x, \lim_{s\to\infty}\bar{\zx}(s)=\zz\}\\
&\qquad =\varliminf_{t\to\infty}\Pr\{\|\hat{\zx}^{u,\bar{\zx}(u)}(t)\|^2\le x\} \Pr\{\lim_{s\to\infty}\bar{\zx}(s)=\zz\}.
\end{align*}
Thus, we conclude that the asymptotic behavior of $\Pr\{\|\bar{\zx}(t)\|^2\le x, \lim_{s\to\infty}\bar{\zx}(s)=\zz\}$ coincides with the asymptotics of $\Pr\{\|\hat{\zx}^{u,\bar{\zx}(u)}(t)\|^2\le x\} \Pr\{\lim_{s\to\infty}\bar{\zx}(s)=\zz\}$. Now we can use~\eqref{eq:eq4} and the fact that $\Pr\{\lim_{s\to\infty}\bar{\zx}(s)=\zz\}\in(0,1]$ to get
\begin{align*}
 \E\{\|\bar{\zx}(t)\|^2\mid \lim_{s\to\infty}\bar{\zx}(s)=\zz\}=O\left(\frac{1}{t}\right) \mbox{ as }t\to\infty.
\end{align*}
\end{proof}

\section{Simulation Results}\label{sec:sim}

{For illustration purposes, let us consider a simple distributed optimization problem:
$$F(z) = F_1(z)+F_2(z)+F_3(z),$$
over a network of $3$ agents. We assume that the function $F_i$ is known only to the agent $i$, $i=1,2,3$, and
$$F_1(z) = \begin{cases}
            (z^3-16x)(z+2), \mbox{ if $|z|\le 10$},\\
            \phantom{-}4248z-32400, \mbox{ if $z > 10$},\\
            -3112z-25040, \mbox{ if $z <-10$}.
           \end{cases}
$$

$$F_2(z)=\begin{cases}
          (0.5z^3+z^2)(z-4), \mbox{ if $|z|\le 10$},\\
            \phantom{-}1620z-12600, \mbox{ if $z > 10$},\\
            -2220z-16600, \mbox{ if $z <-10$}.
         \end{cases}
$$

$$F_3(z)=\begin{cases}
         (z+2)^2(z-4), \mbox{ if $|z|\le 10$},\\
            288z-2016, \mbox{ if $z > 10$},\\
            288z-1984, \mbox{ if $z <-10$}.
         \end{cases}
$$
The plot of the function $F$ on the interval $z\in[-6,6]$ is represented by Figure~\ref{fig:Function}. Outside this interval the function $F$ has no local minima\footnote{Note that the problem $F(z)\to\min_{\R}$ is equivalent to the problem of minimization of the function $F_0(z)=(z^3-16x)(z+2)+(0.5z^3+z^2)(z-4)+(z+2)^2(z-4)$ on the interval [-10, 10], where $F_0(z)$ is extended on the whole $\R$ to meet Assumption~\ref{assum:A1}.}.
\begin{figure}[htb!]
\centering
\includegraphics[width=0.6\linewidth]{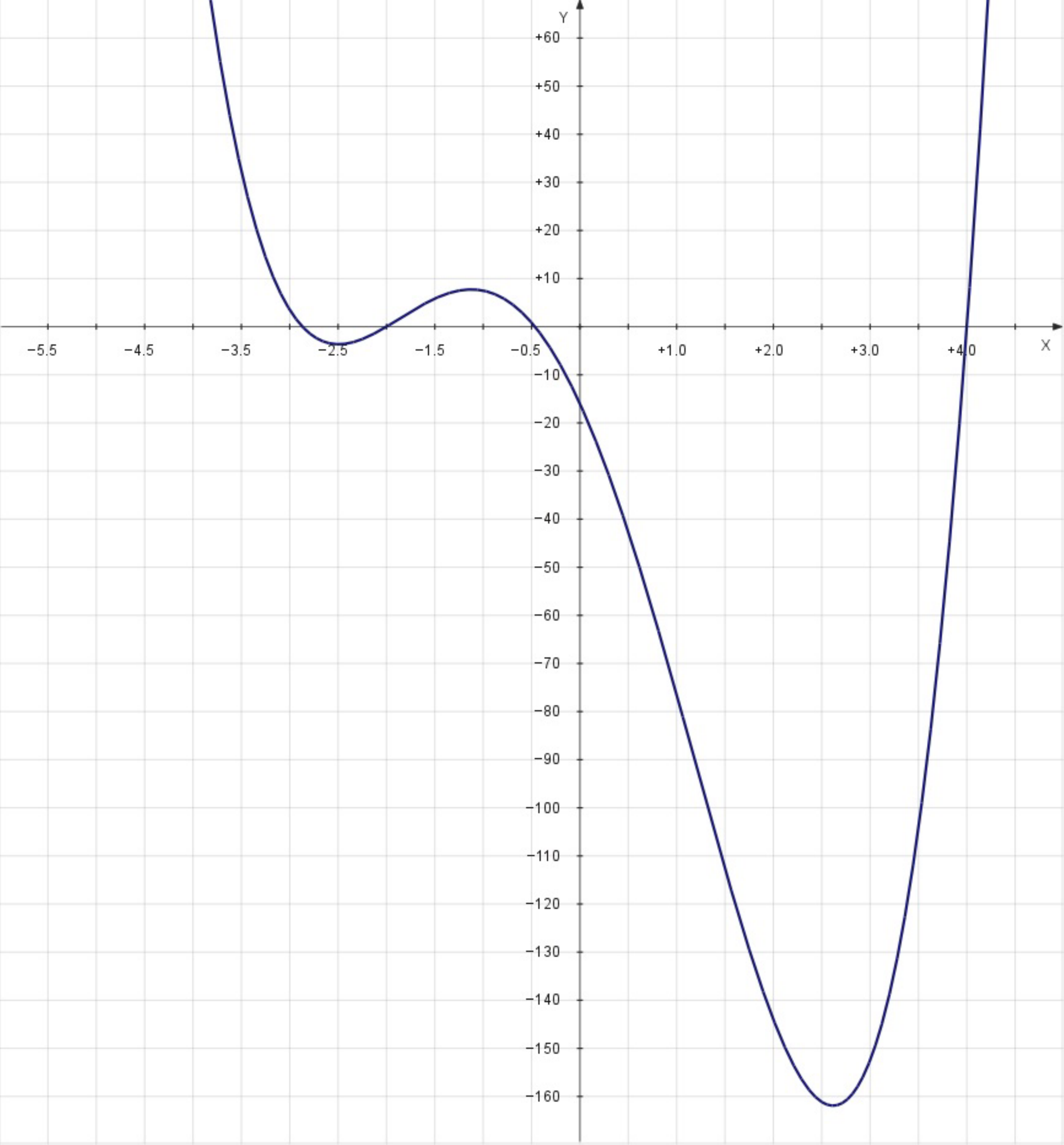}
\caption{Global objective function $F(z)$.}
\label{fig:Function}
\end{figure}

\begin{figure}[htb!]
\centering
\includegraphics[width=0.74\linewidth]{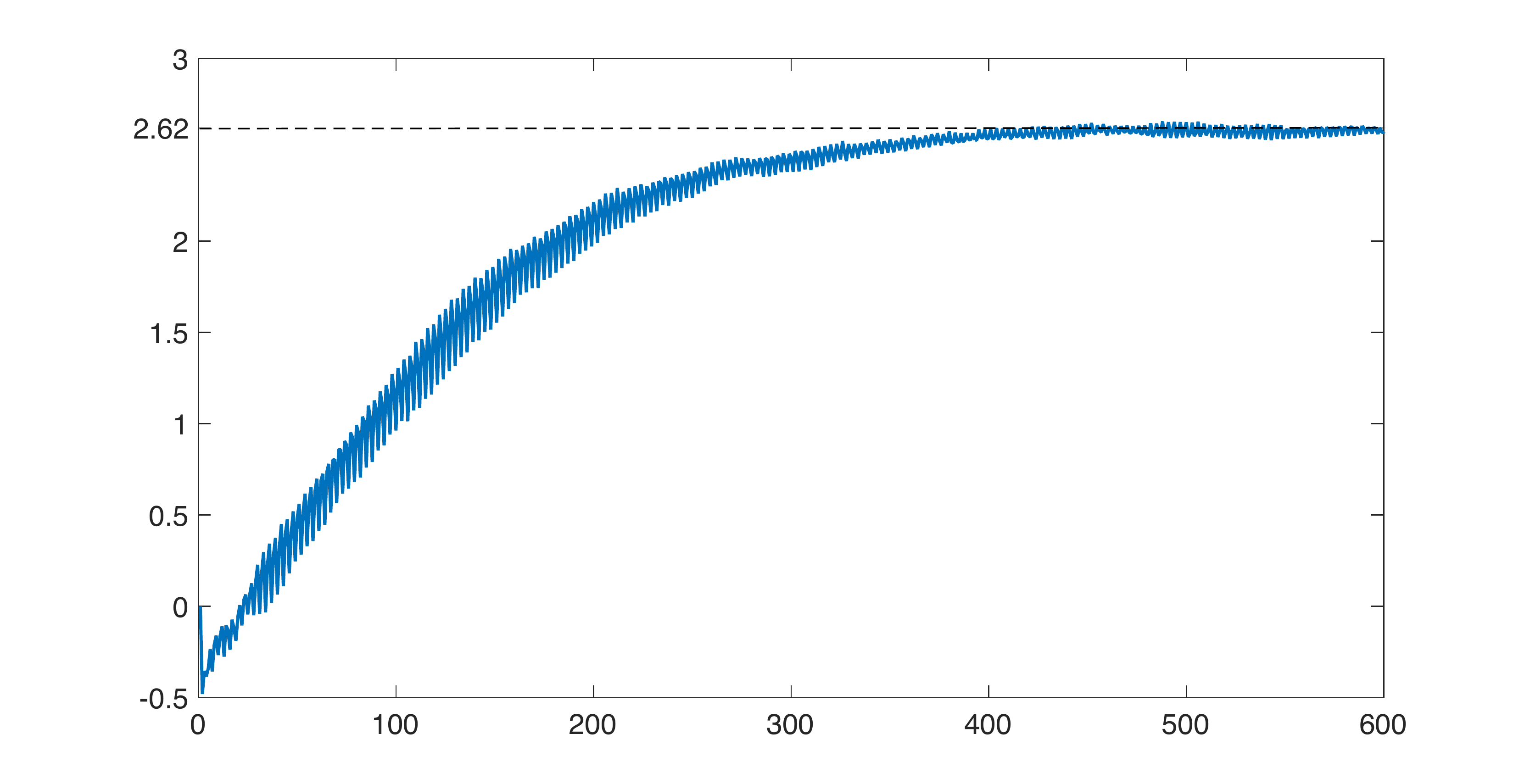}
\caption{The value $z_1(t)$  (given $x_1(0)=0$, $x_2(0)=0$, $x_3(0)=0$).}
\label{fig:1}
\end{figure}

\begin{figure}[htb!]
\centering
\includegraphics[width=0.74\linewidth]{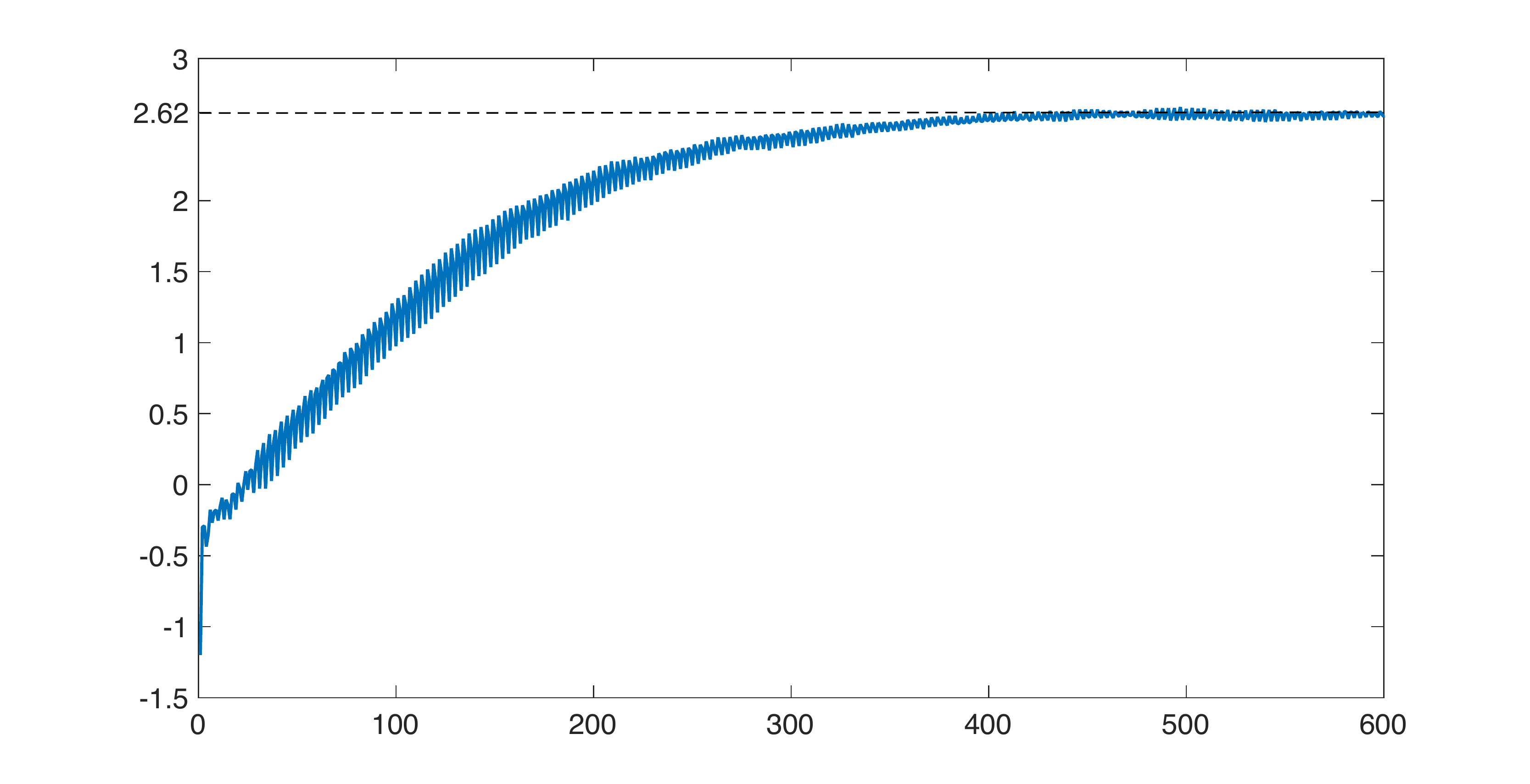}
\caption{The  value $z_2(t)$ (given $x_1(0)=0$, $x_2(0)=0$, $x_3(0)=0$).}
\label{fig:2}
\end{figure}

\begin{figure}[htb!]
\centering
\includegraphics[width=0.74\linewidth]{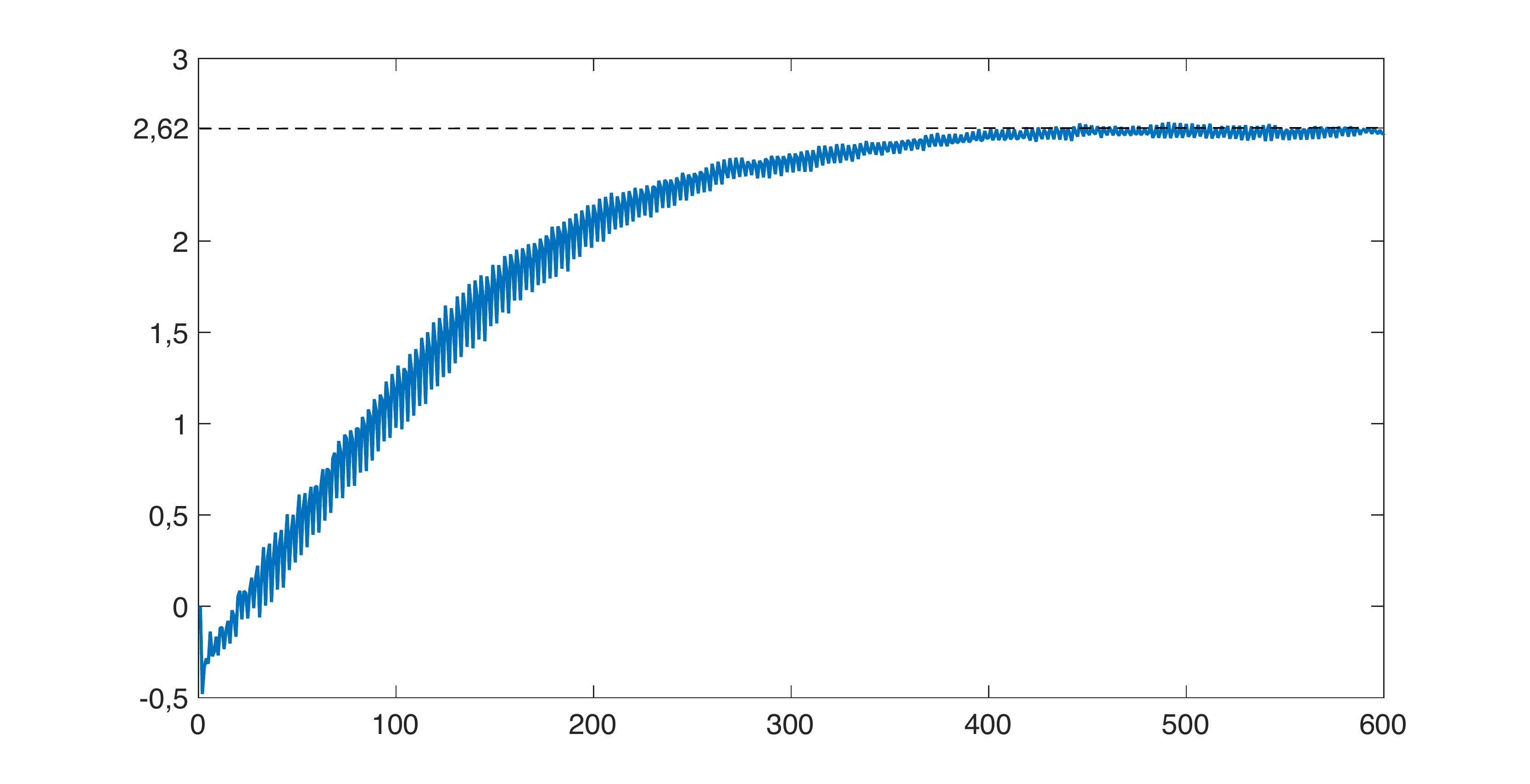}
\caption{The  value $z_3(t)$ (given $x_1(0)=0$, $x_1(0)=0$, $x_1(0)=0$).}
\label{fig:3}
\end{figure}

\begin{figure}[htb!]
\centering
\includegraphics[width=0.73\linewidth]{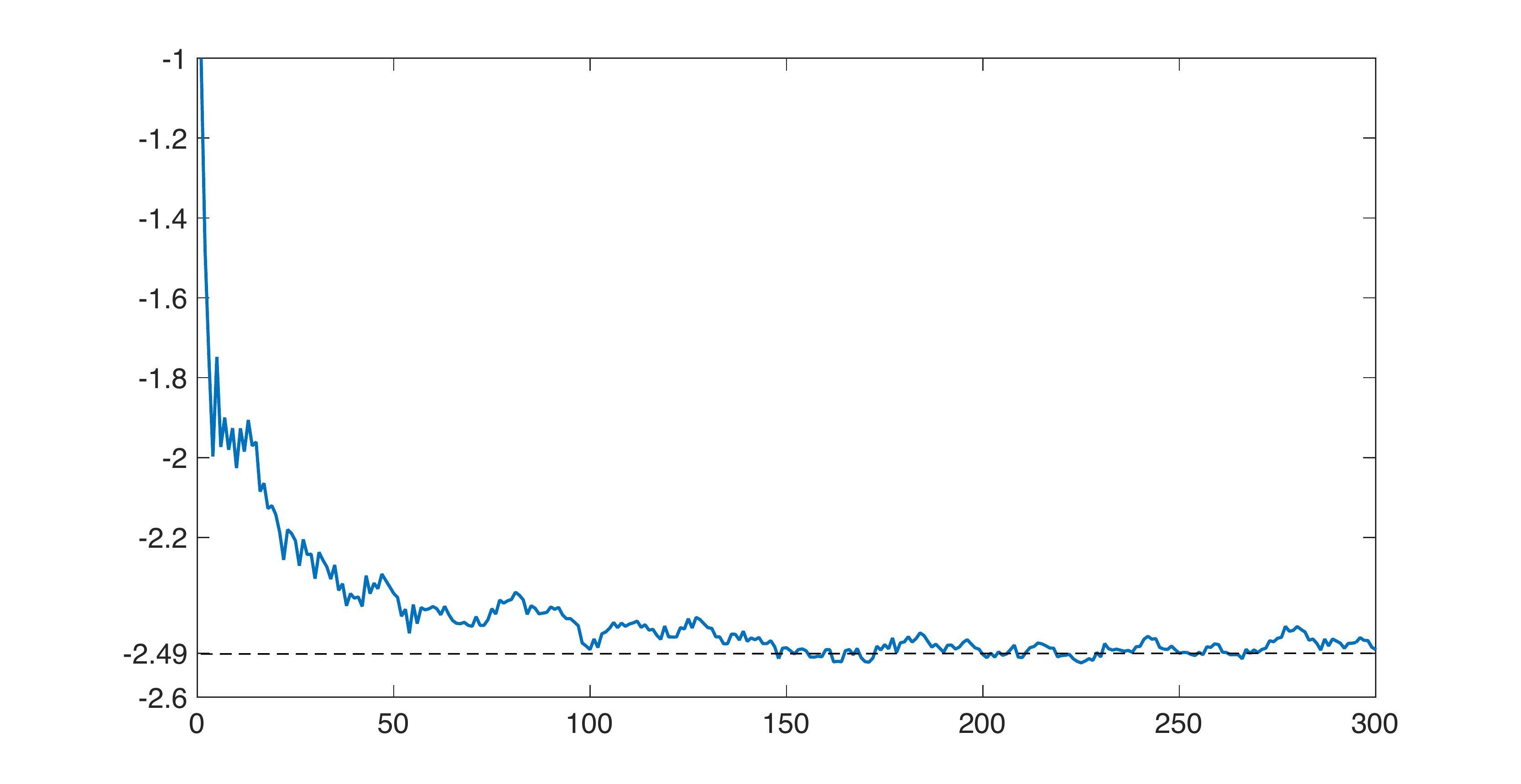}
\caption{The value $z_1(t)$  (given $x_1(0)=-1$, $x_2(0)=-1.2$, $x_3(0) = -1.1$).}
\label{fig:4}
\end{figure}

\begin{figure}[htb!]
\centering
\includegraphics[width=0.73\linewidth]{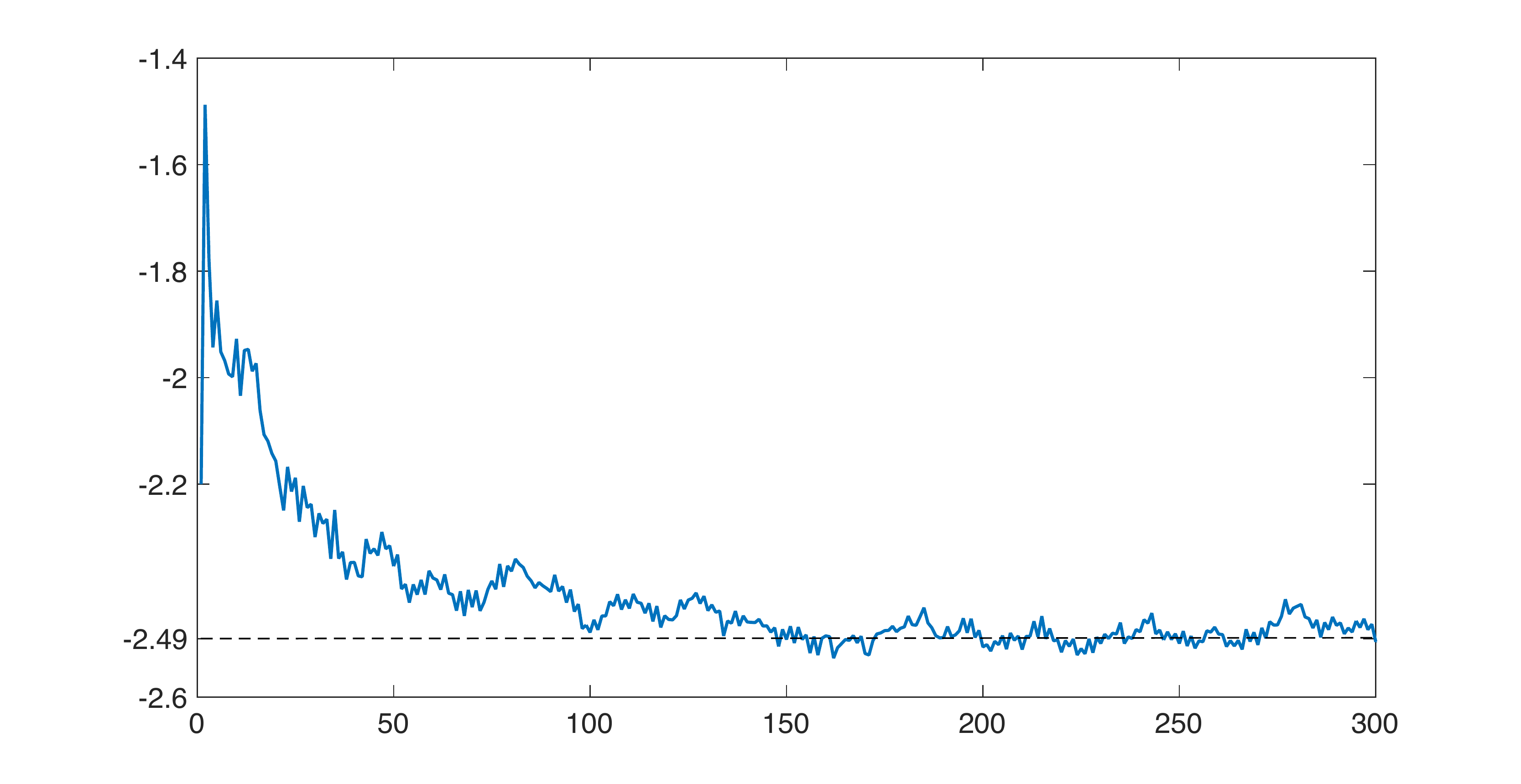}
\caption{The  value $z_2(t)$ (given $x_1(0)=-1$, $x_2(0)=-1.2$, $x_3(0) = -1.1$).}
\label{fig:5}
\end{figure}

\begin{figure}[htb!]
\centering
\includegraphics[width=0.73\linewidth]{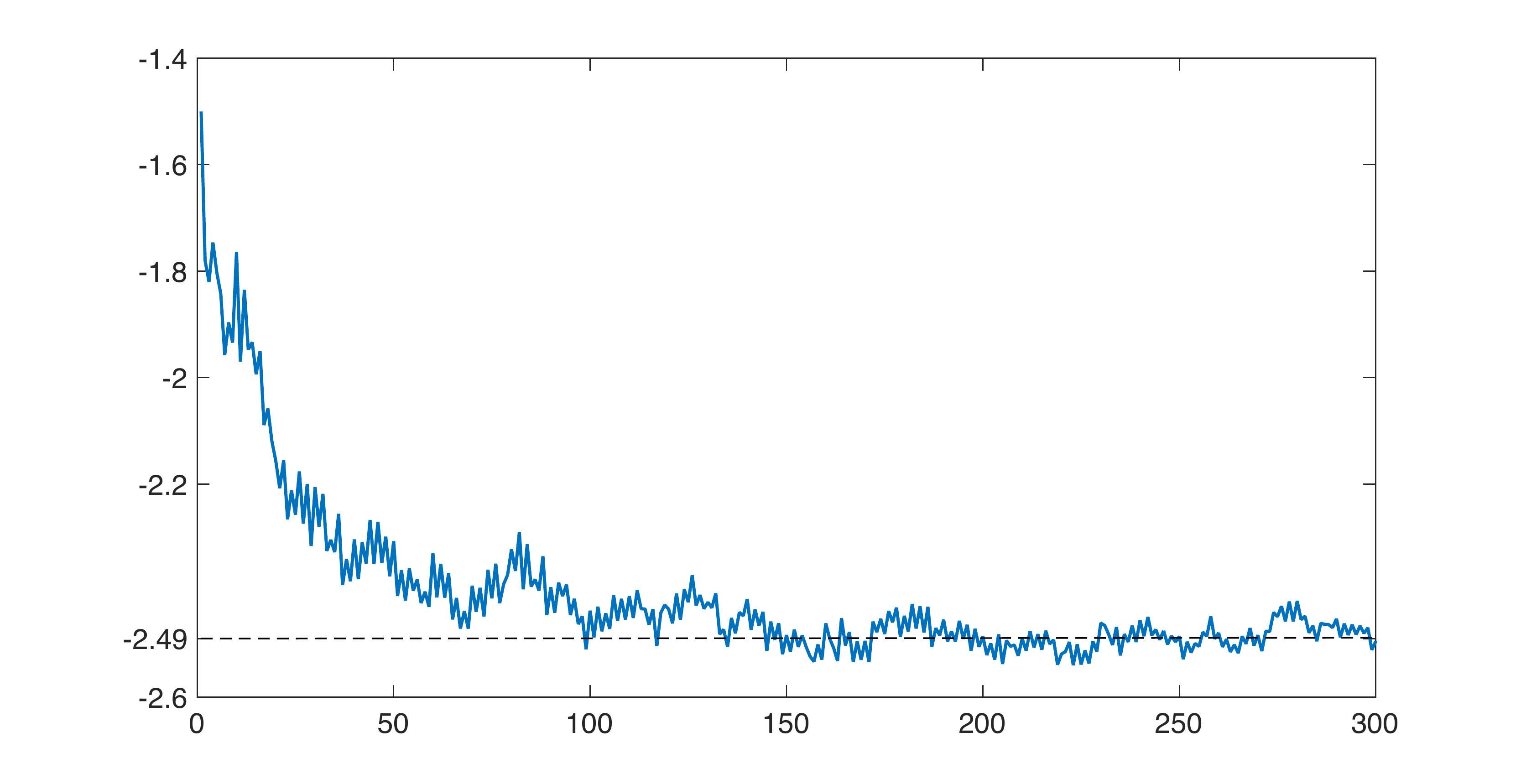}
\caption{The  value $z_3(t)$ (given $x_1(0)=-1$, $x_2(0)=-1.2$, $x_3(0) = -1.1$).}
\label{fig:6}
\end{figure}
For this simulation, the communication graph $G(t)$ is chosen in the following manner: 
the connectivity graph for each two consecutive time slots is chosen to be the following two graph combinations that is randomly set up for the information exchange:
$$([3], 1 \leftrightarrow 2, 2 \leftrightarrow 3)\quad ([3], 1 \leftrightarrow 2, 1\leftrightarrow 3)$$
or
$$([3], 1 \leftrightarrow 2, 1\leftrightarrow 3)\quad ([3], 1 \leftrightarrow 2, 2 \leftrightarrow 3).$$
Obviously, such sequence of communication graphs is $4$-strongly connected.
Moreover, Assumptions~\ref{assum:A1}-\ref{assum:A3} hold for the functions $\{F_i(z)\}_i$ and $F(z)$ and we can apply the perturbed push-sum algorithm to guarantee the almost sure convergence to a local minimum, namely either to $z=-2.49$ or to $z=2.62$. The performance of the algorithm for the initial agents' estimations $x_1(0)=0$, $x_1(0)=0$, $x_1(0)=0$ and $x_1(0)=-1$, $x_2(0)=-1.2$, $x_3(0) = -1.1$ are demonstrated in Figures~\ref{fig:1}-\ref{fig:3} and Figures~\ref{fig:4}-\ref{fig:6}, respectively. We observe that in the first case the algorithm converges to the global minimum $z=2.62$. In the second case, although the initial estimations is close to the local maximum of $F$, $z=-1.12$, the algorithm converges to the
local minimum $z=-2.49$.}
\section*{Acknowledgment}
The authors are grateful for the valuable time and comments of the anonymous reviewers to improve this work. This work was gratefully supported by the German Research Foundation
(DFG) within the GRK 1362 ``Cooperative, Adaptive and Responsive Monitoring of Mixed Mode Environments'' (www.gkmm.de). Behrouz Touri is also thankful to the Air Force Office of Scientific Research for supporting this work under the AFOSR-YIP award FA9550-16-1-0400. 

\section{Discussions and Concluding Remarks}\label{sec:conclude}
In this paper we studied non-convex distributed optimization problems. We demonstrated that under some assumptions on the gradient of the objective function, the known deterministic push-sum algorithm converges to some critical point of the global objective function. However, the deterministic procedure based on the gradient optimization methods cannot distinguish between local maxima and local optima. Motivated by this, we considered the stochastic version of the distributed procedure and proved its almost sure convergence to some local minimum (or to the boundary of one of the connected components of the set of critical points) of the objective function, {in the absence of any saddle point}.
The convergence rate estimations was also obtained for the stochastic process under consideration.

The future analysis may include possible modifications of the algorithm to approach the global minima of the objective function.
Moreover, a payoff-based version of the proposed procedure is also of interest, especially in the application of optimal wind farm control~\cite{BarasW10}.

\bibliographystyle{abbrv}
\bibliography{psum}

\end{document}